\documentclass[draft]{amsart}

\newtheorem{thm}{Theorem}[section]
\newtheorem{cor}[thm]{Corollary}
\newtheorem{lem}[thm]{Lemma}
\newtheorem{prop}[thm]{Proposition}
\theoremstyle{definition}
\newtheorem{defn}[thm]{Definition}
\theoremstyle{remark}

\theoremstyle{definition}

\newcommand{\PP}{\mathbb{P}}

\def\move-in{\parshape=1.75true in 4true in}

\def\twomed{\medskip\medskip}

\usepackage{xypic}

\begin{document}

\title{ Secant Varieties of ${\PP ^1}\times \cdots \times {\PP ^1}$ ($n$-times) are NOT Defective for $n \geq 5$.}

\author{M. V. Catalisano,  A.Geramita, A.Gimigliano}

\address{Maria Virginia Catalisano, DIPTEM - Dipartimento di Ingegneria della Produzione, Termoenergetica e Modelli
Matematici, Piazzale Kennedy, pad. D 16129 Genoa, Italy.}

\email{catalisano@diptem.unige.it}

\address{Anthony V. Geramita,
Department of Mathematics and Statistics, Queen's University,
Kingston, Ontario, Canada,  and Dipartimento di Matematica,
Universit\`a di Genova,Genoa, Italy.}

\email{anthony.geramita@gmail.com}

\address{Alessandro Gimigliano,
Dipartimento di Matematica and CIRAM, Universit\`a di Bologna, 40126
Bologna, Italy.}

\email{gimiglia@dm.unibo.it}

\subjclass{Primary ; Secondary }

\date{}


\begin{abstract}

Let $V_n$ be the Segre embedding of\ ${\PP ^1}\times \cdots \times
{\PP ^1}$\ ($n$ times). We prove that the higher secant varieties
$\sigma_s(V_n)$ always have the expected dimension, except for
$\sigma_3(V_4)$, which is of dimension 1 less than expected.

\end{abstract}

\maketitle 

\section{Introduction}\label{intro}

The problem of determining the dimensions of the higher secant
varieties of embedded projective varieties is a problem which has attracted geometers for over a century.  Original investigations mostly concentrated on the secant line variety and were concerned, largely, with questions of projection.  Varieties which had secant varieties of less than the expected dimension (the so called {\it defective} secant varieties) were especially interesting and were the object of intense study (e.g. all the 2-uple Veronese embeddings of $\PP^n$ into $\PP^{ {n+2\choose 2} -1}$ and all the Segre embeddings of ${\PP^n}\times {\PP^m}$ into $\PP^N$ with $N =(n+1)(m+1)-1)$ (see e.g.\cite{ChCi}, \cite{ChCo}, \cite{K}, \cite{Pa}, \cite{Z} and the bibliographies of these for a sampling of the kinds of results obtained classically).

In more recent times, given the questions raised by computer scientists in complexity theory \cite{BCS} and by biologists and statisticians (see \cite{GHKM} \cite{GSS} \cite{AllRh08}) the focus has shifted to the study of the Segre embeddings of $\PP^{n_1}\times \cdots\times \PP^{n_t}$ for $t \geq 3$.  As is to be expected for a problem with so much interest in such varied disciplines, the approaches to the problem have been varied (see e.g. \cite{BCS} \cite{Land08} for the computational complexity approach, \cite{AllRh07} \cite{GSS}  for the biological statistical approach, \cite{AOP06} \cite {CGG2} \cite{CGG2err}, \cite{CGG5} \cite{CGG6} \cite{ChCi} for the classical algebraic geometry approach, \cite{LaMa08} \cite{LaWe07}  for the representation theory approach, \cite{Draisma} for a tropical approach and \cite{Fr} for a multilinear algebra approach).

Despite all the progress made on the most fundamental question about these secant varieties, namely: how big are they?; many questions remain open.  It is even unknown which secant varieties are defective for the families $\PP^t \times \cdots\times \PP^t$, $n$-times ($t$ fixed, $n \geq 3$) for every $t \geq 1$.  The lovely paper \cite{AOP06} has hazarded a conjecture about this problem.

In this paper we solve this last problem for $t = 1$ (giving solid evidence for the conjecture of Abo-Ottaviani-Peterson).  We show that the only defective secant variety in this infinite family is that of the the secant $\PP^2$'s to $\PP^1\times \PP^1\times \PP^1\times \PP^1 \subset \PP^{15}$ which, instead of forming a hypersurface in $\PP^{15}$, has codimension 2 in $\PP^{15}$ (see Theorem \ref{mainTHM}).

Our approach to this problem has several components, some of which come out of the pioneering work on this sort of question initiated by J. Alexander and A. Hirschowitz \cite{AH95} \cite{AH00} for the Veronese varieties (see also the work of Hartshorne and Hirschowitz \cite{HH85}).  The first, and best known, component of our approach is to reduce the problem (via an application of Terracini's Lemma and polarity) to that of calculating the Hilbert function, in a specific degree, of a non-reduced collection of linear subspaces of some projective space, in fact ``fat" subschemes supported on linear spaces.

The second step, also well-understood (and useful only when one wants to show that the dimension is that ``expected") is to specialize the supports of these fat schemes in a way that permits (usually via Castelnuovo's Lemma) an induction to be carried out.  This is the Alexander-Hirschowitz ``m\'{e}thode d'Horace" strategy, i.e. ``divide and conquer".

It is the specialization that requires some artistry and this ``dividing" that usually gives the most problems.  It is precisely in this ``dividing" part of the procedure that we introduce some new ideas: the key one being to replace the specialized scheme by a new scheme, obtained from the specialized scheme by adding to it some new linear spaces.

These new linear spaces actually show up in the specialized scheme as fixed components (but only for a while, i.e. only for the forms of the particular degree that interest us) which disappear when we consider the ideal of the specialized scheme in higher degree.

This new ``enlarged scheme" (which we are interested in only in one degree) has the effect of ``dividing up a single fat point" into two parts (at least in one degree): one part recognizes the directions that are constrained as a result of our scheme having the new linear components and the other part is what is left.  
This is a ``divide and conquer" strategy on the level of first derivatives and hence also merits the name of a {\it differential Horace method}.  

Thus, on a very fundamental level, what has just been described is our strategy.  Naturally, as was evident in the original papers of Alexander and Hirschowitz on the dimensions of the higher secant varieties to the Veronese varieties, the conversion of the strategy into a proof requires many verifications because of the arithmetic involved.  Indeed, these verifications constitute the major part of what is written here.  We felt it important to lay out, as simply as possible, the nature of the general strategy since the reader can easily get lost in the verifications, which are numerous and tedious.

Now for some quite general considerations.  Let $V \subset \PP^N$ be an irreducible non-degenerate projective variety of dimension $n$.  By a crude parameter count, one expects the dimension of the variety which is the (closure of the) union of all the secant $\PP^{s-1}$'s to $V$ (denoted $\sigma_s(V)$) to be
$$
\min\{ sn+s-1, N \} = \min\{ s(n+1)-1, N \}.
$$
Thus, if one believes that the dimension of $\sigma_s(V)$ is precisely this expected value -- for every $s$ -- then it is enough to verify this

\move-in\noindent $i)$\ only for the integer ${N+1}\over{n+1}$ (assuming that ${N+1}\over{n+1}$ is an integer);

\move-in\noindent $ii)$\ or (when ${N+1}\over{n+1}$ is not an integer) for the two distinct integers
$$
\left\lfloor{ {N+1}\over {n+1} }\right\rfloor \ \ \hbox{ and } \left\lfloor{ {N+1}\over {n+1}} \right\rfloor + 1 = \left\lceil{ {N+1}\over {n+1} }\right\rceil .
$$

In this paper we concentrate on the case
$$
V = V_n:= \PP^1 \times \cdots \times \PP^1 (n-\hbox{times} )
$$
which, via the Segre embedding, is in $\PP^N$ for $N = 2^n - 1$. 

This case has been considered by several authors:

\move-in\noindent $a)$\ In \cite{CGG2} \cite{CGG2err} we showed that when ${N+1}\over {n+1}$ is an integer, i.e. ${2^n}\over {n+1}$ is an integer (so $n = 2^t-1$ for some $t$), then for $s = {2^n\over 2^t} = 2^{{2^t}-(t+1)}$ we have (for every $t \geq 2$) $\sigma_s(V_n) = \PP^N$.  Thus, all the higher secant varieties of $V_n$ ($n = 2^t - 1$, $t \geq 2$) have the expected dimension.  This covers $i)$ above.

\move-in\noindent $b)$\ The cases $n \leq 7$ were covered by ad-hoc methods in \cite{CGG4} as well as in \cite{Draisma} and \cite{AOP06}.

\move-in\noindent $c)$\ In \cite{CGG4} we were able to settle the dimension question (for $\sigma_s(V_n)$ for every $n$ and exactly ONE of 
$$
s = \left \lfloor{ 2^n\over {n+1}}\right\rfloor \ \ \ \ \hbox{or } \ \ \ \  s = \left \lfloor{ 2^n\over {n+1}}\right\rfloor + 1 = \left\lceil{ 2^n\over {n+1}}\right\rceil 
$$
(the one which is even).

The main theorem of this paper is Theorem \ref{mainTHM} which states:

\medskip\noindent
Let $V_n \subset \PP^N$ be the Segre embedding of $\PP^1\times\cdots\times\PP^1$ ($n$-times).  The higher secant varieties, $\sigma_s(V_n)$ have the expected dimension i.e.
$$
\dim  \sigma_s(V_n) = \min \{ s(n+1)-1, N \}
$$
for every $s$ and $n$ EXCEPT for $\sigma_3(V_4)$, which has dimension 13 instead of 14.

\twomed Our main theorem has the following interesting corollary (see e.g. \cite{CGG2} and \cite{BCS} for undefined terms and the correspondence between decomposable tensors and Segre Varieties).

\medskip\noindent
 A generic ${\Bbb C}$-tensor of format
$2\times ...\times 2$ ($n$ times) can be written as the sum of $s$
decomposable tensors for $s= \lceil{2^n\over n+1}\rceil$, and no fewer.

\section{Preliminaries, Notation}\label{prelims}

\bigskip
We will always work over an algebraically closed field $\kappa$, with
char $\kappa =0$.

Let us recall the notion of higher secant varieties.

\begin {defn}\label{defhighersec} Let $V\subseteq \PP ^N$ be a closed
irreducible non-degenerate projective variety of dimension $n$. The $s^{th}$ {\it
higher secant variety} of $X$, denoted $\sigma_s(V)$, is the closure
of the union of all linear spaces spanned by $s$ independent points
of $V$.
\end {defn}

 Recall that there is an inequality involving the dimension of $\sigma_s(V)$.  Namely,
$$
\dim \sigma_s(V)\leq {\rm min} \{N, sn+s-1\} $$
and one ``expects" the inequality should, in general, be an
equality.

When $\sigma_s(V)$ does not have the expected dimension, $V$ is said
to be $(s-1)$-{\it defective}, and the positive integer
$$
\delta _{s-1}(V) = {\rm min} \{N, sn+s-1\}-\dim \sigma_s(V)
$$
is called the  $(s-1)${\it-defect} of $V$.

\bigskip A classical result about higher secant varieties is Terracini's Lemma (see \cite{Te}, \cite{CGG2}):

\begin{prop} \label{Terracini} {\bf Terracini's Lemma:} Let $(V,{\mathcal L})$ be
a polarized, integral, non-singular scheme; if ${\mathcal L}$ embeds
V into $\PP ^N$, then:
$$
T_P(\sigma_s(V)) = \langle T_{P_1}(V),...,T_{P_s}(V)\rangle ,
$$
where $P_1,...,P_s$ are s generic points on V, and P is a generic
point of $\langle P_1,...,P_s\rangle $  (the linear span of $P_1,
\ldots , P_s$); here  $T_{P_i}(X)$ is the projectivized tangent
space of V in $\PP ^N$. {\hfill \qed}
\end{prop}

\bigskip
\begin{defn}\label{fatpoint} Let $V$ be a scheme, and $P\in V$ a closed point; we
define a $m$-fat point, $mP$, to be the subscheme of $V$ defined by
the ideal sheaf ${\mathcal I}^m_{P}\subset {\mathcal O}_V$.
\end{defn}

\medskip

We will often use the notation $m_1P_1+m_2P_2+...+m_sP_s$ to denote
the schematic union of $m_1P_1$,...,$m_sP_s$, i.e. a scheme defined
by the ideal sheaf ${\mathcal I}^{m_1}_{P_1}\cap ...\cap {\mathcal
I}^{m_s}_{P_s}$.

Let $Z\subset V$ be a scheme made of $s$ generic 2-fat points, i.e.
a scheme defined by the ideal sheaf ${\mathcal I}_{Z} = {\mathcal
I}^2_{P_1}\cap ...\cap {\mathcal I}^2_{P_s}\subset {\mathcal O}_V$,
where $P_1,...,P_s$ are $s$ generic points. Since there is a
bijection between hyperplanes of the space $\PP ^N$ containing the
subspace $\langle T_{P_1}(V),...,T_{P_s}(V) \rangle $ and the
elements of $H^0(V,{\mathcal I}_{Z}({\mathcal L}))$, we have:

\begin {cor}\label{secandfat} Let V, ${\mathcal L}$, $Z$, be as
in Proposition \ref{Terracini}; then
$$
\dim \sigma_s(V) = \dim \langle T_{P_1}(V),...,T_{P_s}(V)\rangle = N
- \dim H^0(V,{\mathcal I}_{Z}({\mathcal L})) .
$$
\end{cor}

\medskip\noindent  Now, let $V=\PP ^1\times ... \times \PP ^1$ ($n$-times) and let $V_n \subset \PP ^N$ ($N = 2^t -1$)
be the embedding of $V$ given by ${\mathcal L}={\mathcal
O}_V(1,...,1)$.  By applying the corollary above to our case we get
$$
\dim \sigma_s(V_n) =N-\dim(I_Z)_{(1,1, \ldots , 1)}, 
$$
where $Z\subset  \PP ^1\times ...\times \PP ^1$ is a set of $s$
generic 2-fat points, and $(I_Z) \subset R$ is the multihomogeneous ideal of $Z$ in 
$R = \kappa [x_{0,1},x_{1,1}, \ldots , x_{0,n},x_{1,n}]$,
the ${\Bbb Z}^n$- graded coordinate ring of $\PP ^1\times ...\times \PP ^1$.
\par
Now consider the birational map
$$
g: \PP ^1\times ...\times \PP ^1 ---\rightarrow {\Bbb A}^n,
$$
where:
$$ ((x_{0,1}, x_{1,1}),...,(x_{0,n}, x_{1,n})) \longmapsto
({x_{1,1}\over x_{0,1}}, {x_{1,2}\over x_{0,2}}, \cdots ,
{x_{1,n}\over x_{0,n}}) \ .
$$
This map is defined on the open subset of $\PP ^1\times ...\times
\PP ^1$  given by $\{x_{0,1}x_{0,2}...x_{0,n}\neq 0\}$.

Let $S=\kappa[z_0,\ z_{1,1}, \ z_{1,2}, \ ...\ , z_{1,n} ]$ be the
coordinate ring of $\PP ^n$ and consider the embedding ${\Bbb A}^n
\rightarrow \PP ^n$ whose image is the chart ${\Bbb
A}^n_0=\{z_0=1\}$. By composing the two maps above we get:
$$
f: \PP ^1\times ...\times \PP ^1 ---\rightarrow {\Bbb P}^n,
$$
with
$$
((x_{0,1}, x_{1,1}),...,(x_{0,n}, x_{1,n})) \longmapsto
(1,{x_{1,1}\over x_{0,1}}, {x_{1,2}\over x_{0,2}}, \cdots
,{x_{1,n}\over x_{0,n}} )
$$
$$
= (x_{0,1}x_{0,2} \cdots x_{0,n},\ x_{1,1}x_{0,2} \cdots x_{0,n}, \
\ldots ,\ x_{0,1} \cdots x_{0,n-1}x_{1,n}).
$$

Let $Z\subset \PP ^1\times ...\times \PP ^1$ be a zero-dimensional
scheme which is contained in the affine chart
$\{x_{0,1}x_{0,2}...x_{0,n}\neq 0\}$ and let $Z' = f(Z)$. We want to
construct a scheme $X\subset  \PP ^n$ such that $\dim (I_X)_n=\dim
(I_Z)_{  (1,...,1)}$.

Let $Q_{ 1}, Q_{ 2}, \ldots , Q_{ n}$ be the coordinate points
of $\PP ^n$.    The defining ideal of $Q_i$, ($1 \leq i \leq n$),
is:
$$
I_{Q_i} = (z_0, \ z_{1,1},  \ldots , \widehat {z_{1,i}},   \ldots ,
z_{1,n}) \ .
$$
The following theorem (see \cite{CGG4}) gives the relation
between the homogeneous ideal $I_X$ and the multihomogeneous ideal
$I_Z$:

\medskip

\begin {thm}\label{affineproj}  Let $Z,\ Z'$ be as above and let
$X = Z' + (n-1)Q_1 + ... + (n-1)Q_n\subset \PP ^n$.  Then we have:
$$
\dim (I_X)_n=\dim (I_Z)_{(1,...,1)} \ .
$$
\end {thm}

In particular, we get that if $Z=\emptyset = Z'$, so $X = (n-1)Q_1 + ...
+ (n-1)Q_n$, then $\dim I_X=2^n$, while if $Z'=2P_1+...+2P_s$, then the
expected dimension of $I_X$ is ${\rm max} \{2^n-s(n+1) ; 0 \}$.  It follows that:

\begin{cor} \label{cortoproj} Let $Z\subset \PP ^1\times ...\times \PP ^1$
be a generic set of s 2-fat points and let
 $W = 2P_1 + \cdots + 2P_s + (n-1)Q_1 + \cdots + (n-1)Q_n \subset \PP ^n$. 
  Then we have:
$$
\dim  V_n^s = H(Z,(1,...,1))-1 = (2^n-1) - \dim (I_W)_n.
$$
\end{cor}

\medskip\noindent{\bf Notation:}\  If $X\subset \PP ^n$ we let 
$I_X$ denote the homogeneous ideal of $X$ in the coordinate ring of $\PP
^n$, while if $X\subset Y\subset \PP ^n$, we write $I_{X,Y}$
for the homogeneous ideal of $X$ in the coordinate ring of $Y$.
\par
If $X, \Pi$ are closed subschemes of $\PP^n$, we denote by $Res_\Pi(X)$ the scheme defined by the ideal
$(I_X:I_\Pi)$ and we call it the ``residual scheme" of $X$ with
respect to $\Pi$, while the scheme $Tr_\Pi(X)\subset \Pi$ is the
schematic intersection $X\cap \Pi$, called the ``trace" of $X$ on
$\Pi$.

\bigskip

Let us recall a classical result which we often use:

 \begin{thm} [Castelnuovo] \label{castelnuovo}
Let $\Pi \subseteq \PP ^n$ be a hyperplane, and let $X \subseteq \PP
^n$ be a  scheme. Then
$$
\dim (I_{X, \PP^n})_t  \leq  \dim (I_{ Res_\Pi X, \PP^n})_{t-1}+
\dim (I_{Tr _{\Pi} X, \Pi})_t.
$$
\end{thm}
\medskip

We recall below Theorem 2.3 of \cite{CGG4} which proves ``half" of our Theorem \ref{mainTHM}.  
We rephrase that result as follows:

  \begin{thm} \label{prodottiP1vecchi}  Let  $n, s , x \in \Bbb N$, $n \geq 3$, $x $ even. Let $Q_1,... ,Q_n, P_1,...,P_s $ be generic points in $P^n$. Consider the following scheme
 $$X = (n-1)Q_1+\cdots + (n-1)Q_n + 2P_1+\cdots +2P_s \in \PP^n.$$
 Let
   \[ e =  {\left \lfloor { 2^n \over {n+1}} \right \rfloor }
  \ \ \ , \ \ \  e^* = \left \lceil { 2^n \over {n+1}}  \right  \rceil
   \]
then \par {\rm (i) } if     $s \leq x \leq e $, then
$$\dim (I_X)_n = 2^n -(n+1)s ;$$ \par
{\rm (ii) } if    $s \geq x \geq e^* $, then
$$\dim (I_X)_n =0.$$
    \end{thm}

  \section{Lemmata on certain linear systems}

This section contains several lemmata regarding fat points in
projective space which will be used in the proof of the main theorem
in the next section.

\begin{lem}\label{lemfixcomp}
Let $n, i, m \in \Bbb N$, $ n \geq 2$,  $ m >i$ and $1 \leq i \leq n$. Let $Q_1,... ,Q_{1+i} \in \PP^n$ be  $1+i$ generic points.  Consider
 the linear span $H = <Q_1,...,Q_{i+1}>\cong \PP ^{i} $, and the following scheme of fat points
$$X = m Q_1+\cdots + m Q_{i+1}. $$

{\rm (i)} If $i=n$, then $(I_X)_{m+1}= (0)$;
\par
{\rm (ii)} If $i<n$ then  $H$ is a fixed component,
with multiplicity $m-i$, for the hypersurfaces defined by the forms of $(I_X)_{m+1}$.
\end{lem}

\begin{proof}
Let  $\kappa[ x_{1}, \dots, x_{n+1}]$ be the
homogeneous coordinate ring of $\PP ^n$.
We may assume that the $Q_i$ are coordinate points and so $$I_{Q_j} = (x_1, \dots  ,x_{j-1},x_{j+1}, \dots  , x_{n+1}).$$  Thus,
$$I_X= (x_2, \dots  ,x_{n+1})^m \cap \dots \cap  (x_1, \dots  ,x_{i}, x_{i+2}, \dots  , x_{n+1})^m.$$
{\rm (i)} For $i=n$, we have 
$$I_X= (x_2, \dots  ,x_{n+1})^m \cap \dots \cap 
 (x_1, \dots  ,x_{n-1},  x_{n+1})^m  \cap 
 (x_1, \dots  ,x_{n})^m   .$$
 If  $ f= x_1^{a_1} \cdot   \ldots   \cdot  x_{n+1}^{a_{n+1}}  \in (I_X)_{m+1}$, then
  $ \sum_{j=1} ^{n+1} a_j = m+1 >n+1$. Hence there exists a $j$ with $a_j \geq 2$.
It follows that 
$f \not\in (x_1, \dots  ,x_{j-1}, x_{j+1}, \dots  , x_{n+1})^m  $, a contradiction.

 \par
\noindent
{\rm (ii)} Let  $i\leq n-1$. We have 
$$I_{H} =   (x_{i+2}, \dots  , x_{n+1}),  $$
$$(I_X)_{m+1}= < \{ x_1^{a_1} \cdot  \ldots  \cdot x_{i+1}^{a_{i+1}}
 \cdot  x_{i+2}^{a_{i+2}} \cdot  \ldots   \cdot  x_{n+1}^{a_{n+1}}  \} >$$
where $ \sum_{j=1} ^{n+1} a_j = m+1$, and  $a_j \leq 1$ for $1 \leq j \leq i+1$.
Hence     
$$\deg    (x_1^{a_1} \cdot  \ldots  \cdot x_{i+1}^{a_{i+1}})   \leq i+1 , \hskip1cm 
\deg    (x_{i+2}^{a_{i+2}} \cdot  \ldots   \cdot  x_{n+1}^{a_{n+1}} )   \geq m-i  ,$$
thus
$(I_X)_{m+1} \subset   (x_{i+2}, \dots  , x_{n+1})^{m-i} 
$
 follows.
\end{proof}

\bigskip

From the previous result and Theorem \ref{castelnuovo} we obtain
the following easy, but very useful, remark:
\bigskip

\begin{lem}\label{lemzero}
Let $Y$ be a subscheme of $\PP ^n$, and let $Q_1,... ,Q_n \in \PP^n$
be  generic points.  Consider the scheme theoretic union
 $$X = (n-1)Q_1+\cdots + (n-1)Q_n +Y.$$
 Let $\Pi  \subset  \PP ^{n}$ be a hyperplane through $Q_2,... ,Q_n$, which does not contain $Q_1$;
 let  $W \subset \Pi $ be the projection of  $Res_\Pi X$ into $\Pi$  from $Q_1$, and
 let $\Pi' =  <Q_2,... ,Q_n> \simeq \PP^{n-2}$.
  Let
 $$T= Res_ {\Pi '} (Tr _{\Pi} X ) \subset \Pi .
 $$
Then
$$\dim (I_{X,\PP^n})_n \leq   \dim (I_{W, \Pi})_{n-1} + \dim (I_{T, \Pi})_{n-1} .$$
\end{lem}

\begin{proof} By Theorem \ref{castelnuovo} we know that
$$
\dim (I_{X,\PP^n})_n  \leq  \dim (I_{ Res_\Pi X, \PP^n})_{n-1}+ \dim
(I_{Tr _{\Pi} X, \Pi})_n.
$$
Since $ Res_\Pi X = (n-1)Q_1+(n-2)Q_2+\cdots + (n-2)Q_n + Res_\Pi
Y$, then any form of degree $(n-1)$ in $I_{ Res_\Pi X}$ represents a
cone with $Q_1$ as vertex. It follows that $ \dim (I_{ Res_\Pi X,
\PP^n})_{n-1} =  \dim (I_{W, \Pi})_{n-1} .$

Since $ Tr_\Pi X = (n-1)Q_2+\cdots + (n-1)Q_n + Tr_\Pi Y$, by
Lemma \ref{lemfixcomp} it follows that the linear
space $\Pi'$ spanned by the points $ Q_2,... ,Q_n $ is a fixed
component for the hypersurfaces defined by the forms of 
$(I_{Tr_{\Pi} X, \Pi})_n$. 
Hence $\dim (I_{Tr _{\Pi} X, \Pi})_n = 
\dim (I_{Res _{\Pi ' } {Tr _{\Pi} X, \Pi}})_{n-1}
 = \dim (I_{T, \Pi})_{n-1}$. \par
 So we get
 $$
\dim (I_{X})_n  \leq \dim (I_{W, \Pi})_{n-1} + \dim (I_{T, \Pi})_{n-1}  ,
$$
and we are done.

\end{proof}

The following three lemmata are quite technical, but they will be essential in proving the main theorem.
\par
\begin{lem}\label{duepuntinixpiano}
Let $Y$  be a subscheme of $\PP ^m$, with $m \geq 3$. Let $x \in \Bbb N$.

Choose planes $H_1, \dots, H_x$ in $\PP ^m$ and denote by $R_i$ a generic point on $H_i$. Write 
 $ J^{(1)}_{H_i}$ for the union of two generic simple points with support on $H_i$  ($1 \leq i \leq x$).

Consider the  schemes
 $$X = Y+J^{(1)}_{H_1}+\cdots +J^{(1)}_{H_x}$$
and
  $$X' = Y+2R_1+\cdots +2R_x.$$

If $$\dim (I_{X'})_m =\dim (I_{Y})_m - x(m+1) ,$$ then
 $$\dim (I_{X})_m =\dim (I_{Y})_m - 2x . $$
 \end{lem}

 \begin{proof}
Let $L_i$ be a line through $R_{i}$ lying on $H_i$  ($1 \leq i \leq
x$), and consider the following specialization $\tilde X$ of $X$
 $$\tilde X = Y+Tr _{L_1} 2R_1+\cdots +Tr _{L_x}2R_{x}.$$
 Since $ \tilde X \subset X'$, and  $\dim (I_{X'})_m =\dim (I_{Y})_m - x(m+1) ,$ that is
the $x$ double points $2R_{i}$ impose independent conditions to the
forms of $(I_{Y})_m$, we get
 $$\dim (I_{\tilde X})_m = \dim (I_{Y})_m  - 2x.$$
The conclusion follows by the semicontinuity of the Hilbert
function.
 \end{proof}

\medskip

\begin{lem}\label{lemmaxresiduo}\ (Residue Lemma)
Let  $x,y,m \in \Bbb N$, $m \geq 3$, $0 \leq x \leq \lfloor{{m-1}
\over 2} \rfloor$. Let $Q_1,... ,Q_m$ and $R_{1},...,R_{y}$ be
generic points in $P^m$. Let  $ J^{(2)}_{H_i}$ be the union of two
generic double points $2P_{2i-1}$, $2P_{2i}, $ with support on  a
generic plane $H_i$ through $Q_{2i}$, $Q_{2i+1}$  ($1 \leq i \leq
x$). Consider the following  scheme
 $$
X = (m-1)Q_1+\cdots + (m-1)Q_m +J^{(2)}_{H_1}+ \cdots + J^{(2)}_{H_x} + 2R_{1}+\cdots +2R_{y}.
$$

{\rm (i) } If $x=0$ and $y=0 $, that is
$$X = (m-1)Q_1+\cdots + (m-1)Q_m $$ we have
$$\dim (I_X)_m = 2^m .$$

{\rm (ii) } If $x=1$ and $y=0 $, that is
$$X = (m-1)Q_1+\cdots + (m-1)Q_m +J^{(2)}_{H_1}$$ we have
$$\dim (I_X)_m = 2^m - 2m.$$

{\rm (iii) } If  $x \leq x'$, $y  \leq y' $, and
$$\dim (I_{X'})_m = 2^m - 2mx'-(m+1)y',$$
  then
 $$\dim (I_{X})_m = 2^m - 2mx-(m+1)y,$$
 where
$$X'= (m-1)Q_1+\cdots + (m-1)Q_m +J^{(2)}_{H_1}+ \cdots + J^{(2)}_{H_{x'}} + 2R_1+\cdots +2R_{y'}\subseteq \PP ^m . $$

{\rm (iv) } $$\dim (I_X)_m \geq 2^m - 2mx-(m+1)y.$$

{\rm (v.1) } If  $m =4$, $x=1$, $y=1 $, that is
 $$X=3Q_1+\cdots + 3Q_4 +J^{(2)} _{H_1}+ 2R_1, $$
 then
$$\dim (I_X)_4 = 2^m - 2mx-(m+1)y=3 .$$

{\rm (v.2) } Let $H$ be a  generic plane through $Q_1,Q_4$, let
$J_H^{(1)}=P_1+P_2$ be union of two generic simple points lying on
$H$.   If  $m =4$, $x=1$, $y=1 $, consider the scheme
 $$X'=X+P_1+P_2=3Q_1+\cdots + 3Q_4 +J ^{(2)}_{H_1}+ J^{(1)}_H+2R_1, $$
 then
$$\dim (I_{X'})_4 = 2^m - 2mx-(m+1)y-2=1 .$$

{\rm (vi) } If  $m =5$, $x=1$, $y=3 $,
 that is
 $$X=4Q_1+\cdots + 4Q_5 +J^{(2)}_{H_1}+ 2R_1+\cdots + 2R_3 ,$$
 then
$$\dim (I_X)_5 =  2^m - 2mx-(m+1)y= 4 . $$

{\rm (vii) } If   $x$, $y $ are even,  $0 \leq y \leq \lfloor{{2^{m}
-2mx }\over {m+1}} \rfloor$,
  then
$$\dim (I_X)_m = 2^m - 2mx-(m+1)y. $$
\end{lem}

\begin{proof}
{\rm (i) }  Follows immediately from Theorem \ref{prodottiP1vecchi}.
\par

{\rm (ii) }  By {\rm (i) } we have to prove that a scheme of type
$J^{(2)}$ imposes $2m$ conditions to the forms of
$(I_{(m-1)Q_1+\cdots + (m-1)Q_m})_m $.
\par
Consider the curves in $H_1 \simeq \PP^2$  of degree $m$ , passing
through
$$Tr_{H_1} ((m-1)Q_1+ (m-1)Q_2+2P_1+2P_2).$$
Since the line $Q_1Q_2$ is a fixed component of multiplicity $m-2$,
we have
$$\dim (I_{Tr_{H_1} ((m-1)Q_1+ (m-1)Q_2)})_m =
\dim (I_{ (Q_1+ Q_2)})_2 =6-2=4$$ and
$$\dim (I_{Tr_{H_1} ((m-1)Q_1+ (m-1)Q_2+2P_1+2P_2)})_m = 0.$$

Hence the two double points $Tr_{H_1}2P_1$ and $Tr_{H_1}2P_2$ lying
on $H_1$ give  at most 4 conditions, instead of 6,  to the forms of
$ (I_{Tr_{H_1} ((m-1)Q_1+ (m-1)Q_2)})_m$. It follows that the two double points
 $2P_1$ and $2P_2$ of $\PP^m$, give at most
$2(m+1)-2=2m$ conditions to the forms of $(I_{(m-1)Q_1+ (m-1)Q_2})_m
$, and so  to the forms of $(I_{(m-1)Q_1+\dots + (m-1)Q_m})_m $
also. In other words,  $\dim(I_{X})_m\geq 2^m-2m.$
\par
Now we will prove that $\dim(I_{X})_m\leq 2^m-2m$ by induction on
$m$. For $m=3$ we have $X= 2Q_1+2Q_2+ 2Q_3 +J^{(2)}_{H_1},$ and
since the plane $H_1$ is a fixed component for the surfaces defined
by the forms of  $(I_X)_3$, we easily get  $\dim (I_X)_3 =  2^3-6$.

 Assume $m>3$.   Let $\Pi  \subset  \PP ^{m}$ be the hyperplane through $H_1, Q_2,... ,Q_m$ (note that $Q_1 \notin \Pi$).
 Let $W $ be the projection of  $Res_\Pi X$ into $\Pi$  from $Q_1$, hence
  $$W = (m-2)Q_2+\cdots + (m-2)Q_m+R_1+R_2  \subset \Pi \simeq \PP^{m-1}$$
  where $<Q_2,Q_3,R_1,R_2>$ is the plane $H_1$.
  Since by Theorem \ref{prodottiP1vecchi} one double point imposes independent conditions to the forms of  $(I_{(m-2)Q_2+\cdots + (m-2)Q_m})_{m-1}$, by Lemma \ref{duepuntinixpiano} we get
  $$\dim (I_W)_{m-1} = 2^{m-1} - 2. $$
 Let $\Pi' =  <Q_2,... ,Q_n> \simeq \PP^{n-2}$, and  let $T= Res_ {\Pi '} (Tr _{\Pi} X )$, so
 $$T = (m-2)Q_2+\cdots + (m-2)Q_m+ Tr _{\Pi} J^{(2)}_{H_1} \subset \Pi \simeq \PP^{m-1} .$$ 

 \par
 By the induction hypothesis we have
  $$\dim (I_T)_{m-1} = 2^{m-1} - 2(m-1). $$
Thus by Lemma  \ref{lemzero} , we get
$$\dim (I_{X})_m \leq 2^{m-1} - 2+2^{m-1} - 2(m-1)=2^m-2m,$$
and we are done.
\par

{\rm (iii) } and {\rm (iv) } easily follow from (i) and (ii).
\par

{\rm (v.1) } Note that the line $L$ $=Q_1R_1$ is a fixed component
for the hypersurfaces defined by the forms of $(I_{X})_4$, hence
$\dim(I_{X})_4=\dim(I_{Z})_4 $, where $Z=X+L$.
 Let $\Pi  \subset  \PP ^{4}$ be the hyperplane through $P_1,P_2, Q_2,Q_3,Q_4$ (note that it doesn't contain $Q_1$) and let $R_1'= L \cap \Pi$. We want to apply Lemma \ref{lemzero} to $Z$ and $\Pi$.
\par
 The projection of  $Res_\Pi Z$ into $\Pi$  from $Q_1$ is
 $$W=2Q_2+2Q_3 + 2Q_4 +P_1+P_2+ 2R'_1 \subset \Pi  \simeq \PP^3$$
 If $\Pi' $ is the plane $ <Q_2,Q_3,Q_4>$, we have
 $$T= Res_ {\Pi '} (Tr _{\Pi} Z )=2Q_2+2Q_3 + 2Q_4 +J^{(2)}_{H_1}+ R'_1 \subset \Pi  \simeq \PP^3.$$
 \par
 Since 5 double  points impose independent conditions to the surfaces of degree $3$
 in $\PP^3$,  we have (by Lemma \ref{duepuntinixpiano}) that 
  $$\dim (I_W)_3=
  \dim (I_{2Q_2+2Q_3 + 2Q_4 + 2R'_1})_3-2 = 2.$$
 Since $ R'_1$ is a generic point in $ \Pi $, by (ii) we get
   $$\dim (I_T)_3=2^3-2 \cdot3 -1=1.$$
Hence by  Lemma \ref{lemzero} we have
$$\dim(I_{X})_4=\dim(I_{Z})_4 \leq  \dim (I_W)_3+\dim (I_T)_3=3$$
and the conclusion follows from (iv).
\par

{\rm (v.2) } We proved this case by an ad-hoc specialization of
$X'$ (that we leave to the reader) and by direct computations using
CoCoA  (see \cite{cocoa}).

{\rm (vi) } By (iv) it suffices to prove that $\dim (I_{X})_5 \leq
4$.
\par
 Observe that the lines $L_1$ $=Q_1R_1$ and  $L_2$ $=Q_1R_2$ are a fixed component for the hypersurfaces defined by the forms of $(I_{X})_5$, hence $\dim(I_{X})_5=\dim(I_{X+L_1+L_2})_5 $.

Let $\Pi  \subset  \PP ^{5}$ be the hyperplane  $<P_1,P_2,
Q_2,Q_3,Q_4,Q_5>$,  that does not contain $Q_1$, and let $R_1'= L_1
\cap \Pi$ and $R_2'= L_2 \cap \Pi$.
\par
Let  $Z$ be the scheme obtained from $X+L_1+L_2$ by  specializing
the point $R_3$ on $\Pi$. If we prove that $\dim(I_{Z})_5 \leq 4$,
then (by the semicontinuity of the Hilbert function) we have $\dim(I_{X})_5 \leq 4$, and we are done.
In order to show that   $\dim(I_{Z})_5 \leq 4$, recall that, by
Lemma \ref{lemzero},
  $$\dim (I_{Z})_5 \leq   \dim (I_{W, \Pi})_{4} + \dim (I_{T, \Pi})_{4} ,$$
 where
 $$W= 3Q_2+3Q_3+3Q_4+3Q_5+P_1+P_2+2R'_1+2R'_2+R_3 $$
$$ = 3Q_2+3Q_3+3Q_4+3Q_5+J^{(1)}_{H_1}+2R'_1+2R'_2+R_3 \subset \Pi \simeq \PP^4$$
 and
 $$T=3Q_2+3Q_3+3Q_4+3Q_5+2P_1+2P_2+R'_1+R'_2+2R_3  $$
$$=3Q_2+3Q_3+3Q_4+3Q_5+J^{(2)}_{H_1}+R'_1+R'_2+2R_3   \subset \Pi \simeq \PP^4.$$
By (v.1), since $Q_2, Q_3, P_1,P_2$  lye on $H_1$ and $R'_1, R'_2$
are simple generic points, we have
$$\dim  (I_{T})_{4} = 1.$$
\par
Unfortunately, since a double point does not impose independent
conditions to the hypersurfaces defined by the forms of
$(I_{3Q_2+3Q_3+3Q_4+3Q_5+2R'_1+2R'_2})_4$, we can not use Lemma
\ref{duepuntinixpiano} to compute $\dim  (I_{W})_{4} $. So denote by
$\Lambda \simeq \PP^3 $ the linear span of $Q_2, Q_3, Q_4, P_1,P_2$
and specialize $R'_1$ and $R_3$ on it. The line $L=Q_5R_2'$ is a
fixed component for the  hypersurfaces defined by the forms of $
(I_{W})_{4} $, hence $\dim  (I_{W})_{4} =\dim  (I_{W+L})_{4} $. Now
$$Tr_{\Lambda} (W+L) = 3Q_2+3Q_3+3Q_4+J^{(1)}_{H_1}+2R'_1+R''_2+R_3 , $$
where $R''_2 = \Lambda \cap L$.  Hence, since the plane
$<Q_2,Q_3,Q_4>$ is a fixed component, we get (by Lemma
\ref{duepuntinixpiano}) that
$$\dim (I_{Tr_{\Lambda} (W+L)})_4 =
\dim (I_{2Q_2+2Q_3+2Q_4+J^{(1)}_{H_1}+2R'_1+R''_2+R_3})_3 =0 .
$$
It follows that  $\dim (I_{W})_{4} = \dim (I_{Res _\Lambda W})_{3}$,
where
$$Res _\Lambda W = 2Q_2+2Q_3+2Q_4+3Q_5+R'_1+2R'_2 .$$
Thus  the forms of  $(I_{Res _\Lambda W})_{3}$ define cones with
$Q_5$ as vertex, and so we easily get $\dim (I_{W})_{4} = \dim
(I_{Res _\Lambda W})_{3}=3.$
\par
Finally, by  Lemma \ref{lemzero} we have
 $$\dim (I_{Z})_5 \leq   \dim (I_{W})_{4} + \dim (I_{T})_{4}=4 .$$

 \par
{\rm (vii) } Theorem \ref{prodottiP1vecchi} covers the case $x=0$,
so assume that $x \geq2$.  By (iv) it suffices to prove that $\dim
(I_{X})_m  \leq 2^m -2mx - (m+1)y $.
\par

Let $\Pi  \subset  \PP ^{m}$ be a hyperplane through $Q_2, \dots,
Q_m$, not containing $Q_1$. Let $\tilde X$ be the scheme obtained by
specializing, onto $\Pi$, the planes $H_1, \dots, H_{x\over 2} $ and
the points $R_1, \dots, R_{y\over 2} $.

Since the lines $L_i= Q_1P_i$ ($1 \leq i \leq 2x$) and $M_j =
Q_1R_{j}$ ($1 \leq j \leq y$)
 are  fixed component for the hypersurfaces defined by the forms of $(I_{ \tilde X})_m$,
we have  $$\dim (I_{ \tilde X})_m= \dim (I_{\tilde X+L_1+ \dots
+L_{2x}+M_1+ \dots +M_{y}})_m .$$ Let
$$Z = \tilde X+L_1+ \dots +L_{2x}+M_1+ \dots +M_{y}.$$
Since by semicontinuity $\dim (I_{X})_m \leq \dim (I_{\tilde X})_m$,
if we prove that
$$\dim (I_{Z})_m \leq   2^m -2mx - (m+1)y $$
we are done.

Set $P'_{i}= L_i \cap \Pi$ and  $R'_{j}= M_j \cap \Pi$. Note  that
  $P'_{i}=P_i$ for $i=1,\dots,x$ and  $R'_{j}=R_j$ for $j=1,\dots,{y \over 2}$,
  and recall  that
$$<Q_{2i},Q_{2i+1},P_{2i-1},P_{2i}>=H_i \simeq \PP^2 \ \ \ \ \ \ {\rm for} \ \ 1 \leq i \leq {x\over2}$$
$$<Q_{2i},Q_{2i+1},P'_{2i-1},P'_{2i}>=H'_{i} \simeq \PP^2 \ \ \ \ \ \ {\rm for} \ \ {{x\over2}+1} \leq i \leq {x}$$
where $H'_{i}$ is the projection of  $H_i $ into $\Pi$  from $Q_1$.
\par
By Lemma \ref{lemzero} we have
$$\dim (I_{Z,\PP^m})_m \leq   \dim (I_{W, \Pi})_{m-1} + \dim (I_{T, \Pi})_{m-1} ,$$
where
$$W = (m-2)Q_2+\cdots + (m-2)Q_m +
P_1+\dots+P_x+2P'_{x+1}+\dots+2P'_{2x} $$
$$+R_{1}+\cdots +R_{y \over 2}
+2R'_{{y \over 2}+1}+\cdots +2R'_{y }$$
$$ = (m-2)Q_2+\cdots + (m-2)Q_m + J^{(1)}_{H_1}+\dots+J^{(1)}_{H_{x \over 2}}
+ J^{(2)}_{H'_{{x \over 2}+1}}+\dots+J^{(2)}_{H'_{x}} $$
$$+R_{1}+\cdots +R_{y \over 2}
+2R'_{{y \over 2}+1}+\cdots +2R'_{y };$$

$$T = (m-2)Q_2+\cdots + (m-2)Q_m +
2P_1+\dots+2P_x+P'_{x+1}+\dots+P'_{2x} $$
$$+2R_{1}+\cdots +2R_{y \over 2}
+R'_{{y \over 2}+1}+\cdots +R'_{y }$$
$$ = (m-2)Q_2+\cdots + (m-2)Q_m +
 J^{(2)}_{H_1}+\dots+J^{(2)}_{H_{x \over 2}}+J^{(1)}_{H'_{{x \over 2}+1}}+\dots+J^{(1)}_{H'_{x}} $$
$$+2R_{1}+\cdots +2R_{y \over 2}
+R'_{{y \over 2}+1}+\cdots +R'_{y }.$$

Since clearly  $W \simeq T$, it suffices to compute $ \dim
(I_{W})_{m-1}$.

We work by induction on $m$. The first case  is $m=5$, $x=y=2$. We
want to prove that
$$\dim (I_{Z})_m \leq   2^m -2mx - (m+1)y = 32-20-12=0$$

In this case we have
$$W = 3Q_2+3Q_3+3Q_4+3Q_5 +P_1+P_2+2P'_{3}+2P'_{4} +R_{1}+2R'_{2}$$
$$ = 3Q_2+3Q_3+3Q_4+3Q_5 + J^{(1)}_{H_1}+J^{(2)}_{H'_{2}}+R_{1}+2R'_{2}\subset \Pi \simeq \PP^4. $$
Since $R_1$ is a generic simple point, by (v.2) we get $ \dim
(I_{W})_{4} =0$, and the conclusion follows from Lemma
\ref{lemzero}.

Assume $m>5$.  Set
$$\Theta = (m-2)Q_2+\cdots + (m-2)Q_m,$$
and let $Y$ be the scheme obtained from $W$ by cutting off the
simple points  $R_{1}+\cdots +R_{y \over 2}$, that is

$$Y = W -(R_{1}+\cdots +R_{y \over 2})$$
$$ =\Theta+ J^{(1)}_{H_1}+\dots+J^{(1)}_{H_{x \over 2}}
+ J^{(2)}_{H'_{{x \over 2}+1}}+\dots+J^{(2)}_{H'_{x}} +2R'_{{y \over
2}+1}+\cdots +2R'_{y }.$$ Let $S_{1}\in H_1,\dots,S_{x \over 2}\in
H_{x \over 2}$ be generic points. Now we define some other schemes
related to $Y$, whose dimensions are computable using the induction 
hypothesis, mostly by substituting the two points of $J^{(1)}_{H_i}$
with a double point $2S_i$.
\par
We consider four cases:

\begin{enumerate}
 \item{\it ${x\over2}$ even; ${y\over2}$ even}.
Let
\par
$Y_{(1)} =  \Theta
 + J^{(2)}_{H'_{{x \over 2}+1}}+\dots+J^{(2)}_{H'_{x}}
  +2S_{1}+\cdots +2S_{x \over 2}
+2R'_{{y \over 2}+1}+\cdots +2R'_{y }.$

 \noindent
Hence in $Y_{(1)}$ we have $x' ={{x} \over2}$ schemes of type
$J^{(2)}$ and $y' ={{x+y} \over 2}$ double points.

\item  {\it ${x\over2}$ even ;  ${y\over2}$ odd .}
 Let
 \par
$Y_{(2)} =  \Theta + J^{(2)}_{H'_{{x \over
2}+1}}+\dots+J^{(2)}_{H'_{x}} +2S_{1}+\cdots +2S_{x \over 2}
+2R'_{{y \over 2}+1}+\cdots +2R'_{y } +2P,$
\par\noindent
where $2P$ is an extra generic double point.  Hence in $Y_{(2)}$ we
have $x' ={{x} \over2}$ schemes of type $J^{(2)}$ and $y'  =({{x+y}
\over 2}+1)$ double points.

\item  {\it ${x\over2}$ odd ;  ${y\over2}$ even. }
 Let  \par
$Y_{(3)} =  \Theta+J^{(2)}_{H_{x \over 2}}+ J^{(2)}_{H'_{{x \over
2}+1}}+\dots+J^{(2)}_{H'_{x}} +2S_{1}+\cdots +2S_{{x \over 2}-1}
+2R'_{{y \over 2}+1}+\cdots +2R'_{y }.$
\par\noindent
Note that  $J^{(1)}_{H_{x \over 2}} \in Y$ has been replaced by
$J^{(2)}_{H_{x \over 2}}$.  Hence in $Y_{(3)}$ we have $x' ={{x}
\over2}+1$ schemes of type $J^{(2)}$ and $y'  =({{x+y} \over 2}-1)$
double points.

\item  {\it ${x\over2}$ odd ;  ${y\over2}$ odd. }
Let \par \noindent $Y_{(4)} =  \Theta +J^{(2)}_{H_{x \over 2}}+
J^{(2)}_{H'_{{x \over 2}+1}}+\dots+J^{(2)}_{H'_{x}} +2S_{1}+\cdots
+2S_{{x \over 2}-1} +2R'_{{y \over 2}+1}+ \cdots +2R'_{y }+2P,$
\par\noindent
where $2P$ is an extra generic double point. Note that
$J^{(1)}_{H_{x \over 2}} \in Y$ has been replaced by
$J^{(2)}_{H_{x \over 2}}$, hence in $Y_{(4)}$ we have $x' ={{x}
\over2}+1$ schemes of type $J^{(2)}$ and $y'  ={{x+y} \over 2}$ double
points.
\end{enumerate}

Observe that in all the four cases above $x' $ and $y' $ are even. Now
we need to verify that if $x$, $y$ are even, $0 \leq x \leq
\lfloor{{m-1} \over 2} \rfloor$ and
 $0 \leq y \leq \lfloor{{2^{m} -2mx }\over {m+1}} \rfloor$, then
\par
(a) $0 \leq x'  \leq \lfloor{{m-2} \over 2} \rfloor$;
\par
(b) $0 \leq y' \leq \lfloor{{2^{m-1} -2(m-1) x' }\over {m}} \rfloor$.

The following table may help:

$$
\begin{matrix}
  x \over 2  & & y \over 2  & &   x'   &   y'    \\
& & & \\
  even & & even &  &  {{x} \over2}       &  {{x+y} \over 2}  \\
  & & & \\
  even & & odd   & &  {{x} \over2}       &  {{x+y+2} \over 2}   \\
  & & & \\
  odd   & & even & &  {{x+2} \over2}    &  {{x+y-2} \over 2}    \\
  & & & \\
  odd   & & odd   & & {{x+2} \over2}     &  {{x+y} \over 2}   \\
  \end{matrix}
$$

\par
(a)  $ x' \geq 0$ is obvious. For $m=6$ we have $x = 2$, hence  $x'  =
{x\over 2}+1=2 \leq  \lfloor{{m-2} \over 2} \rfloor =2.$
\par
Assume $m \geq 7$. Since $ x'  \leq {x\over 2}+1$, it is enough to
prove that ${x\over 2}+1\leq \lfloor{{m-2} \over 2} \rfloor$. Now
${x\over 2}+1\leq \lfloor{{m-2} \over 2} \rfloor \iff {x\over
2}+1\leq {{m-2} \over 2} \iff  m-4-x \geq 0.$
\par
Since  $x \leq \lfloor{{m-1} \over 2} \rfloor$ implies  $2x \leq
m-1$, we get
$$ m-4-x \geq m-4-{ {m-1 }\over 2 }={ {m-7 }\over 2 } \geq 0 ,$$
and (a) is verified.
\par
(b) Obviously $y'  \geq 0$.\par {\it Cases} (1) {\it  and} (2) : {\it
${x\over2}$  even}. \par In these case $x' = {x \over 2 } $, $y'  =
{{x +y}\over 2 } $ for ${y\over2}$  even, and $y'  = {{x +y}\over 2
}+1 $  for ${y\over2}$ odd.

Since ${x \over 2 } \neq 0$ is even, we have $ \lfloor{{m-1} \over
2} \rfloor \geq x \geq 4$, hence $m \geq 9$.

\par
For $m=9$ we have $x=4$,  $x' =2$, $y \leq  \lfloor{{2^{m} -2mx
}\over {m+1}} \rfloor = 44$ , $\left \lfloor{{2^{m-1} -2(m-1)x'
}\over {m}}\right  \rfloor =24$ and it is easy to see che $y' \leq
24$.
\par
Assume $m \geq 10$. It suffices to show that
$${{x +y}\over 2 }+1 \leq \left \lfloor{{2^{m-1} -2(m-1)x'  }\over {m}}\right  \rfloor.$$

Now
$${{x +y}\over 2 }+1 \leq \left \lfloor{{2^{m-1} -2(m-1)x' }\over {m}}\right  \rfloor \iff
m(x+y+2) \leq  2^{m} -2(m-1)x   $$
$$ \iff  2^m-  x(3m-2) - my- 2m  \geq 0 .$$
Since $y \leq{{2^{m} -2mx }\over {m+1}} $, we get \par $2^m-
x(3m-2) - my- 2m  $ \par $\geq  2^m-  x(3m-2) - m{{2^{m} -2mx }\over
{m+1}}- 2m  $\par $={1\over {m+1}} (2^m(m+1)-  x(3m-2)(m+1) -
m(2^{m} -2mx ) -2m(m+1) ) $\par $={1\over {m+1}} (2^m-
x(m^2+m-2)-2m(m+1) ) $ \par $\geq  {1\over {m+1}} (2^m-  {{m-1}
\over 2} (m^2+m-2)-2m(m+1) ) $ \par $=  {1\over {2m+2}} (2^{m+1}-
m^3-4m^2-m-2 ).$ \par

Since for $m \geq 10$ it is easy to verify that  $(2^{m+1}-
m^3-4m^2-m-2 ) \geq 0$, we are done.

{\it Cases} (3) {\it  and} (4): {\it ${x\over2}$  odd}. \par In
these cases $x' = {x \over 2 } +1$, $y' = {{x +y}\over 2 } -1$ for
${y\over2}$  even, and $y' = {{x +y}\over 2 } $  for ${y\over2}$
odd.

\par
For $m=6, 7, 8, 9$ we have $x=2$,  $x' =1$, and

$$
\begin{matrix}
 m  & &  \lfloor{{2^{m} -2mx }\over {m+1}} \rfloor   & &
 \left \lfloor{{2^{m-1} -2(m-1)x' }\over {m}}\right  \rfloor     \\
& & & \\
6 & &  5 &  & 3        \\
  & & & \\
7 & & 12   & &  7       \\
  & & & \\
8  & & 24 & &  14     \\
  & & & \\
9   & & 47   & &26      \\
  \end{matrix}
$$
 and it is easy to verify that $y' \leq 2, 6, 12, 24,$ respectively.
\par
Assume $m \geq 10$. It suffices to show that
$${{x +y}\over 2 } \leq \left \lfloor{{2^{m-1} -2(m-1)x' }\over {m}}\right  \rfloor.$$
Now
$${{x +y}\over 2 } \leq \left \lfloor{{2^{m-1} -2(m-1)x' }\over {m}}\right  \rfloor \iff
m(x+y) \leq  2^{m} -2(m-1)(x+2)   $$
$$ \iff  2^m-  (3m-2) x - my- 4m +4 \geq 0 .$$
Since $y \leq{{2^{m} -2mx }\over {m+1}} $, we get \par $ 2^m-
(3m-2) x - my- 4m +4  $ \par $\geq  2^m-  (3m-2) x - m {{2^{m} -2mx
}\over {m+1}} - 4m +4 $\par

$={1\over {m+1}} (2^m(m+1)-  (3m-2)(m+1)x - m(2^{m} -2mx ) - 4m(m+1)
+4(m+1) ) $\par $={1\over {m+1}} (2^m-  x(m^2+m-2)-4m^2+4 ) $ \par
$\geq  {1\over {m+1}} (2^m-  {{m-1} \over 2} (m^2+m-2)-4m^2+4 ) $
\par $=  {1\over {2m+2}} (2^{m+1}- (m-1)(m^2+9m+6)).$ \par

Since for $m \geq 10$ it is easy to verify that  $(2^{m+1}-
(m-1)(m^2+9m+6)) \geq 0$, we are done.

 Since the two conditions, (a) and (b), are verified, we may compute the dimension of
$(I_{Y_{(i)}})_{m-1}$ using the induction hypothesis . In particular we wish to remark that
 the schemes of type $J^{(2)}$ and the double points of each $Y_{(i)}$  impose the expected number of conditions to $(I_\Theta)_{m-1}$.
Using Lemma \ref{duepuntinixpiano}  so we obtain the dimension of
$(I_Y)_{m-1}$  in each of the four cases above, that is:
$$\dim(I_Y)_{m-1}= \dim(I_\Theta)_{m-1} - 2 \left ({x\over 2} \right )
- 2(m-1) \left ({x\over 2} \right ) -m \left({y\over 2} \right ) $$
$$ =2^{m-1}  - mx-m\left({y\over 2} \right ) .
$$
Since $W$ is simply the union of $Y$ and ${y\over 2}$ generic simple points,
and since $2^{m-1}  - mx-m \left({y\over 2} \right ) \geq
\left({y\over 2} \right ) $ as $ y \leq \lfloor{{2^{m} -2mx }\over
{m+1}} \rfloor$, we have
$$\dim(I_W)_{m-1}= \dim(I_Y)_{m-1} - \left({y\over 2} \right )=
2^{m-1}  - mx-m\left({y\over 2} \right ) - \left({y\over 2} \right ).
$$

It now follows that 
$$\dim (I_{Z})_m
\leq   \dim (I_{W})_{m-1} + \dim (I_{T})_{m-1} = 2 \dim
(I_{W})_{m-1} $$
$$= 2^{m}  - 2mx-(m+1) y.
$$

  \end{proof}

\begin{lem}\label{lemmaxtraccia}\ (Trace Lemma)

Let  $x,y,m \in \Bbb N$,   $m \geq 4$, $0 \leq x \leq \lfloor{{m-1}
\over 2} \rfloor$. Let $Q_1,... ,Q_m$ and $R_{1},...,R_{y}$ be
generic points in $\PP^m$ and let $H_i \cong \PP^2$ be  a generic plane
through $Q_{2i}$, $Q_{2i+1}$ $(1 \leq i \leq x)$. Let  $X  \subseteq
\PP ^m$ be the following scheme
$$X = (m-1)Q_1+\cdots + (m-1)Q_m +H_1+ \cdots + H_x + 2R_1+\cdots +2R_y.$$

{\rm (i) } If $x=1$ and $y=0 $, that is
$$X = (m-1)Q_1+\cdots + (m-1)Q_m +{H_1}$$ we have
$$\dim (I_X)_m = 2^m - 4.$$

{\rm (ii) } If  $x \leq x'$, $y \leq y' $, and
$$\dim (I_{X'})_m = 2^m - 4x'-(m+1)y',$$
  then
 $$\dim (I_{X})_m = 2^m - 4x-(m+1)y,$$
 where
$$X'= (m-1)Q_1+\cdots + (m-1)Q_m +{H_1}+ \cdots + {H_{x'}} + 2R_1+\cdots +2R_{y'}\subseteq \PP ^m . $$

{\rm (iii) } $$\dim (I_X)_m \geq 2^m - 4x-(m+1)y.$$

{\rm (iv) } If  $m =4$, $x=1$, $y=2 $, that is
 $$X=3Q_1+\cdots + 3Q_4 +{H_1}+ 2R_1+2R_2, $$
 then
$$\dim (I_X)_4 = 2^m - 4-(m+1)y=2 .$$

{\rm (v) } If  $m =5$, $x=2$, $y=4 $,
 that is
 $$X=4Q_1+\cdots + 4Q_5 +{H_1}+H_2+ 2R_1+\cdots + 2R_4 ,$$
 then
$$\dim (I_X)_5 =  2^m - 4x-(m+1)y= 0 . $$

{\rm (vi) } If   $x$, $y $ are even,  $0 \leq y \leq \lfloor{{2^{m}
-4x }\over {m+1}} \rfloor$,
  then
$$\dim (I_X)_m = 2^m - 4x-(m+1)y. $$

\end{lem}

\begin{proof}
{\rm (i) }  By Lemma \ref{lemmaxresiduo}{\rm (i) } we have to prove
that a plane imposes $4$ conditions to the forms of
$(I_{(m-1)Q_1+\cdots + (m-1)Q_m})_m $.
\par
Consider the curves of degree $m$ in $H_1 \simeq \PP^2$ passing
through
$$Tr_{H_1} ((m-1)Q_1+ (m-1)Q_2).$$
Since the line $Q_1Q_2$ is a fixed component of multiplicity $m-2$,
we have
$$\dim (I_{Tr_{H_1} ((m-1)Q_1+ (m-1)Q_2)})_m =
\dim (I_{ (Q_1+ Q_2)})_2 =6-2=4  .$$ Hence a plane imposes at most
4 conditions  to the hypersurfaces defined by the forms of $
(I_{Tr_{H_1} ((m-1)Q_1+ (m-1)Q_2)})_m$, and so also to the hypersurfaces 
defined by the forms of $(I_{(m-1)Q_1+\dots + (m-1)Q_m})_m $. In
other words,  $\dim(I_{X})_m\geq 2^m-4.$
\par
Now we will prove, by induction on $m$, that $\dim(I_{X})_m\leq 2^m-4$. For $m=3$ we have $X= 2Q_1+2Q_2+ 2Q_3 +{H_1},$ and since the
plane $H_1$ is a fixed component for the surfaces defined by the
forms of  $(I_X)_3$, we easily get  $\dim (I_X)_3 =  2^3-4$.  

Assume $m>3$.   Let $\Pi  \subset  \PP ^{m}$ be the hyperplane through $H_1, Q_2,... ,Q_m$ (not containing $Q_1$).
 Let $W $ be the projection of  $Res_\Pi X$ into $\Pi$  from $Q_1$, hence
  $$W = (m-2)Q_2+\cdots + (m-2)Q_m \subset \Pi \simeq \PP^{m-1}.$$
  By Lemma \ref{lemmaxresiduo}{\rm (i) } we have
  $$\dim (I_W)_{m-1} = 2^{m-1} . $$
 Let $\Pi' =  <Q_2,... ,Q_m> \simeq \PP^{m-2}$, and  let $T= Res_ {\Pi '} (Tr _{\Pi} X )$, so
 $$T = (m-2)Q_2+\cdots + (m-2)Q_m+ {H_1} \subset \Pi \simeq \PP^{m-1} .$$
  \par
 By the induction hypothesis we have
  $$\dim (I_T)_{m-1} = 2^{m-1} - 4. $$
Thus by Lemma \ref{lemzero} , we get
$$\dim (I_{X})_m \leq 2^{m-1} +2^{m-1} - 4=2^m-4,$$
and we are done.
\par

{\rm (ii) } and {\rm (iii) } easily follow from (i) and Lemma
\ref{lemmaxresiduo}{\rm (i) }.
\par

{\rm (iv) } Since the lines $L_1$ $=Q_1R_1$ and $L_2$ $=Q_1R_2$ are
fixed components of the hypersurfaces defined by the forms of
$(I_{X})_4$, we have $\dim(I_{X})_4=\dim(I_{Z})_4 $, where
$Z=X+L_1+L_2$.   Let $\Pi  \subset  \PP ^{4}$ be the hyperplane
 $ <Q_4, H_1>=<Q_2,Q_3,Q_4, H_1>$ and observe that $Q_1 \not\in \Pi $.  
Let  $R_1'= L_1 \cap \Pi$ and  $R_2'= L_2 \cap \Pi$.
 Now we want to apply Lemma \ref{lemzero} to $Z$ and $\Pi$.

 The projection of  $Res_\Pi Z$ into $\Pi$  from $Q_1$ is
 $$W=2Q_2+2Q_3 + 2Q_4 +2R'_1+2R'_2 \subset \Pi  \simeq \PP^3$$
 If $\Pi' $ is the plane $ <Q_2,Q_3,Q_4>$, we have
 $$T= Res_ {\Pi '} (Tr _{\Pi} Z )=2Q_2+2Q_3 + 2Q_4 +{H_1}+ R'_1+R_2' \subset \Pi  \simeq \PP^3.$$
 \par
 Since 5 double  points impose independent conditions to the surfaces of degree $3$
 in $\PP^3$  we have
  $$\dim (I_W)_3=0.$$
 Since $ R'_1$ and  $ R'_2 $ are generic points in $ \Pi $, by (i) we get
   $$\dim (I_T)_3=2^3-4-2=2.$$
Hence by  Lemma \ref{lemzero} we have
$$\dim(I_{X})_4=\dim(I_{Z})_4 \leq  \dim (I_W)_3+\dim (I_T)_3=2$$
and the conclusion follows from (iii).
\par


{\rm (v) } Let $\Pi  \subset  \PP ^{5}$ be the hyperplane  $<
Q_2,Q_3,Q_4,Q_5,H_1,>$,  that does not contain $Q_1$, Let  $\tilde X
$ be the scheme obtained from $X$ by  specializing the points $R_1$
and $R_2$ on $\Pi$.
  Observe that the lines $L_i$ $=Q_1R_i$ are fixed components for the hypersurfaces defined by the forms of $(I_{\tilde X})_5$, hence
  $\dim(I_{\tilde X})_5=\dim(I_{Z})_5 $, where
  $$Z= \tilde X+L_1+\dots+L_4.$$
 If we prove that $\dim(I_{Z})_5 =0$ then,
 by the semicontinuity of the Hilbert function, we will have $\dim(I_{X})_5=0$, and we will be done.
 \par
To that end, let $R_3'= L_3 \cap \Pi$,  $R_4'= L_4 \cap \Pi$, and let $H'_{2}$ be
the projection of  $H_2$ into $\Pi$  from $Q_1$. By   Lemma
\ref{lemzero}
  $$\dim (I_{Z})_5 \leq   \dim (I_{W, \Pi})_{4} + \dim (I_{T, \Pi})_{4} ,$$
 where
 $$W= 3Q_2+3Q_3+3Q_4+3Q_5+H'_2+R_1+R_2+2R'_3+2R'_4  \subset \Pi \simeq \PP^4$$
 and
 $$T=3Q_2+3Q_3+3Q_4+3Q_5+H_1+2R_1+2R_2+R'_3+R'_4    \subset \Pi \simeq \PP^4.$$
By (iv), since $R_1,R_2,R'_3,R'_4$  are simple generic points, we
have
$$\dim  (I_{W})_{4} = \dim  (I_{T})_{4} = 0,$$
from which $\dim(I_{Z})_5= 0$ follows.
 \par

{\rm (vi) } Theorem \ref{prodottiP1vecchi} covers the case $x=0$, so
assume that $x \geq2$.  By (iii) it suffices to prove that $\dim
(I_{X})_m  \leq 2^m -4x - (m+1)y $.
\par

Let $\Pi  \subset  \PP ^{m}$ be a hyperplane through $Q_2, \dots,
Q_m$, not containing $Q_1$. Let $\tilde X$ be the scheme obtained by
specializing, onto $\Pi$, the planes $H_1, \dots, H_{x\over 2} $ and
the points $R_1, \dots, R_{y\over 2} $.

Since the lines $L_i = Q_1R_{i}$ ($1 \leq i \leq y$)
 are  fixed components for the hypersurfaces defined by the forms of $(I_{ \tilde X})_m$,
we have  $$\dim (I_{ \tilde X})_m= \dim (I_{\tilde X+L_1+ \dots
+L_{y}})_m .$$ Let
$$Z = \tilde X+L_1+ \dots +L_{y}.$$
Since (by semicontinuity) $\dim (I_{X})_m \leq \dim (I_{\tilde X})_m$,
if we can prove that
$$\dim (I_{Z})_m \leq   2^m -4x - (m+1)y $$
we will be done.

We begin by setting $R'_{i}= L_i \cap \Pi$. Note  that
  $R'_{i}=R_i$ for $i=1,\dots,{y \over 2}$.
 Let $H'_{i}$ be the projection of  $H_i $ into $\Pi$  from $Q_1$.
By Lemma \ref{lemzero} we have
$$\dim (I_{Z,\PP^m})_m \leq   \dim (I_{W, \Pi})_{m-1} + \dim (I_{T, \Pi})_{m-1} ,$$
where
$$W = (m-2)Q_2+\cdots + (m-2)Q_m +
H'_{{x \over 2}+1}+ \cdots +H'_{x} +R_{1}+\cdots +R_{y \over 2}
+2R'_{{y \over 2}+1}+\cdots +2R'_{y }.$$
$$T = (m-2)Q_2+\cdots + (m-2)Q_m +
H_1+ \cdots +H_{x \over 2}+2R_{1}+\cdots +2R_{y \over 2} +R'_{{y
\over 2}+1}+\cdots +R'_{y } . $$

Clearly  $W \simeq T$, moreover $R'_{{y \over 2}+1}+\cdots +R'_{y }$
are generic simple points, hence it suffices to compute $ \dim
(I_{Y})_{m-1}$, where
$$Y= \Theta+ H_1+ \cdots +H_{x \over 2}+2R_{1}+\cdots +2R_{y \over 2} \subset \Pi \simeq \PP^{m-1}$$
and
$$\Theta = (m-2)Q_2+\cdots + (m-2)Q_m,$$
Now we are ready to prove (vi), which we do by induction on $m$ (with $x \geq 2$). The first
case to consider  is $m=5$, $x=2$, $y \leq 4$.  That case follows immediately from (ii) and (v).
\par
For $m=6$, by (ii) we need only consider the case $x=2$, $y=8$ for which we 
have to show that   $ \dim (I_{Z})_{6}=0$.  Now
$$Y = \Theta + H_1+2R_1+ \dots +2R_4\subset \Pi \simeq \PP^5$$
hence by (v) we get $  \dim (I_{Y})_{5} = 4$, so $  \dim
(I_{W})_{5}=\dim (I_{T})_{5} = 0$, and
 $  \dim (I_{Z})_{6}=0$ follows.

Now assume $m\geq 7$. We want to compute  $\dim (I_{Y})_{m-1} $ by induction. To this end, if ${y \over 2}$ is odd, add
a double point to $Y$, and if ${x \over 2}$  is odd also, add  a
generic plane through $Q_{m-1}, Q_m$ (this is possible since $x+1 <
m-1$). Hence we have to check that:

(a) $ {x \over 2}+1  \leq \lfloor{{m-2} \over 2} \rfloor$;
\par
(b) $  {y \over 2}+1  \leq \lfloor{{2^{m-1} -4x' }\over {m}}
\rfloor$, where $x' = {x \over 2}$ if $ {x \over 2}$ is even, and
$x' = {x \over 2}+1$ if $ {x \over 2}$ is odd. 

For (a) see the proof
of Lemma \ref{lemmaxresiduo}(vii).
\par
(b) If $ {x \over 2}$ is even, \par $  {y \over 2}+1  \leq \left
\lfloor{{2^{m-1} -4x' }\over {m}} \right \rfloor \iff m {y }+2m
\leq    2^{m} - 4 x \iff
  2^{m} - 4 x -my-2m \geq 0.
$ 

Since
 $y \leq {{2^{m} -4x}\over {m+1}} $ and $x \leq {{m-1} \over 2} $
 we get
$$ 2^{m} - 4 x -my-2m \geq
2^{m} - 4 x -m{{2^{m} -4x}\over {m+1}} -2m = {1\over {m+1}} (2^{m} -
4 x -2m(m+1))
$$
$$ \geq
{1\over {2m+2}} (2^{m+1} - 4 (m-1) -4m(m+1)) = {1\over {2m+2}}
(2^{m+1}  -4m^2-8m+4)),
$$
and $2^{m+1}  -4m^2-8m+4 \geq 0$  for $m\geq 7  $.
\par

 If $ {x \over 2}$ is odd, \par
$  {y \over 2}+1  \leq \left \lfloor{{2^{m-1} -4x' }\over {m}}
\right \rfloor \iff m {y }+2m  \leq    2^{m} - 4 x -8 \iff
  2^{m} - 4 x -my-2m -8\geq 0.
$ 

Since
 $y \leq {{2^{m} -4x}\over {m+1}} $ and $x \leq {{m-1} \over 2} $
 we get
$$ 2^{m} - 4 x -my-2m-8 \geq
2^{m} - 4 x -m{{2^{m} -4x}\over {m+1}} -2m -8
$$
$$=
{1\over {m+1}} (2^{m} - 4 x -2m(m+1)-8(m+1))
 \geq
{1\over {2m+2}} (2^{m+1}  -4m^2-24m-12)),
$$
and $2^{m+1}  -4m^2-24m-12 \geq 0$  for $m\geq 8  $. For $m=7$, we
have $x=2$ and $y \leq 14$ hence
$$2^{m} - 4 x -my-2m-8 \geq 128-8-98 -14-8 =0,
$$
so (b) is proved.
\par
Now we can compute $\dim (I_{Y})_{m-1} $. We get
$$\dim (I_{Y})_{m-1} = 2^{m-1}  - 4 \cdot {x\over 2}   - m \cdot {y\over 2}  ,$$
and from here
$$\dim (I_{T})_{m-1} =\dim (I_{W})_{m-1} =
2^{m-1} - 4 \cdot {x\over 2}   - m\cdot {y\over 2} 
-   {y\over 2} ,$$ hence
$$\dim (I_{Z})_m \leq   \dim (I_{W})_{m-1} + \dim (I_{T})_{m-1}=
2^m - 4x - (m+1)y,$$ and the lemma is proved.
\end{proof}

\section{The main theorem} \label{mainresult}

\bigskip
Now we come to the proof of the main theorem of this paper. 

\begin{thm}\label{mainTHM}
Let  $n, s  \in \Bbb N$, $n \geq 3$.  Let
$$
V_n = \PP^1\times \cdots \times \PP^1 \ \ \ \ (\hbox{$n$-times}).
$$

The dimension of $\sigma_s(V_n) \subset \PP^N$ \ ($N = 2^n -1$), is always the expected dimension, i.e.
$$
\dim \sigma_s(V_n) = \min \{ N, s(n+1)-1 \}
$$
for all $n$, $s$ as above EXCEPT for $n = 4, s=3$.

Moreover, the dimension of $\sigma_3(V_4) = 13$ (rather than 14, as expected).
\end{thm}

Using the results of Section \ref{prelims} we observe that  proving Theorem \ref{mainTHM} is
equivalent to proving:
\bigskip

\begin{thm} \label{mainTHM2}
  Let  $n, s  \in \Bbb N$, $n \geq 3$, and let $Q_1,... ,Q_n, P_1,...,P_s $ be generic points in $\PP ^n$. 
  Consider the following schemes
 $$X = (n-1)Q_1+\cdots + (n-1)Q_n + 2P_1+\cdots +2P_s \subset \PP ^n.$$
 When $(n,s)\neq (4,3)$ set
   \[ e =  {\left \lfloor { 2^n \over {n+1}} \right \rfloor } \hbox{ and }
  \ \ \ \ \ \  e^* = \left \lceil { 2^n \over {n+1}}  \right  \rceil = e+1.
   \]
Then:
\par {\rm (i) } if    $s \leq e $, we have 
$$\dim (I_X)_n = 2^n -(n+1)s ,$$

\par
{\rm (ii) } if   $s \geq e^* $, we have $$\dim (I_X)_n =0.$$ 

{\rm (iii) } If $(n,s)=(4,3)$, we have $\dim (I_X)_4 = 2$.
 \end{thm}

\medskip

\begin{proof}
When $e$, respectively $e^*$, is even see Theorem \ref{prodottiP1vecchi}, and the
same happens  when ${ 2^n \over {n+1}}$ is an
integer (necessarily even); for $(n,s)=(4,3)$ see \cite[Example 2.2]{CGG4}.
So we have only  to deal with the case when $e$ or $e^*$  are
odd, and it suffices to prove the statements (i) and (ii)
for $s=e$ and  $s=e^*$,  respectively.
\par

 Let 
 $$e=2t+1,\ \ e^*=2t^*+1,\ \ n=4q+r, \ \ 0\leq r < 4 .$$
 We 
consider  a specialization $\tilde X$ of $X$, which is defined as follows:
\begin{itemize}
\item
for all $1\leq i\leq q$, let $\Lambda_i =
<Q_1,Q_{2i},Q_{2i+1},P_{2i-1}> \cong \PP ^3$; we specialize $P_{2i}$
to a generic point of $\Lambda_i$;

\item
let $\Pi$ be a generic hyperplane through $Q_2,...,Q_n$ and
specialize the ${{s-1}\over 2}$  points $P_{2q+1},...,P_{2q+{{s-1}\over 2}}$ on $\Pi$.
Notice that 
$
{{s-1}\over 2} = \left \{
\begin{matrix}
 t & {\rm for}& s =e \\
 t^* & {\rm for}& s =e^* \\
 \end{matrix}
  \right.
 .$
\end{itemize}
Such specializations are possible since  $2q+1 \leq n$,
$ 2q+t \leq e$ and $2q+t^* \leq e^* $.

Now, by Lemma \ref{lemfixcomp}, it easily follows that the linear spaces $\Lambda_i$ are
fixed components for the hypersurfaces of $(I_{\tilde X})_n$ as are the lines $L_j$, where $L_j$ is the line through $Q_1$ and $P_j$, for $1\leq j \leq s$ . 
 Hence we obviously have that $(I_{\tilde X})_n = (I_{Z})_n$, where
$$ Z = (n-1)Q_1+\cdots + (n-1)Q_n + 2P_1+\cdots +2P_s + \Lambda _1
+ ...+ \Lambda_q + L_{2q+1} + ...+ L_{s},$$ (notice that the lines
$L_1$,...,$L_{2q}$ are already contained in the $\Lambda_i$'s).

\medskip

In case $s=e$, we surely have $\dim (I_X)_n \geq 2^n -(n+1)e$, and,
by semicontinuity, $\dim(I_{\tilde X})_n \geq \dim (I_X)_n $.  Since
we know that $\dim (I_{\tilde X})_n= \dim (I_{Z})_n$, we only have to prove
 that $\dim (I_{Z})_n \leq 2^n -(n+1)e$. 
\par
In case  $s=e^*$,
since $\dim (I_{\tilde X})_n \geq \dim (I_X)_n $, we only have
to prove that $\dim (I_{\tilde X})_n =\dim (I_{Z})_n=0$.

  By using Lemma \ref{lemzero} on $(I_Z)_n$ we have that
 $$\dim (I_{Z})_n \leq  \dim (I_{W,\Pi})_{n-1} +\dim (I_{T,\Pi})_{n-1}, $$
 where   
 $$W=  (n-2)Q_2 +...+ (n-2)Q_n + J_{H_1}^{(2)} + ... + J_{H_q}^{(2)} 
+ P'_{2q+1}+... + P'_{2q+{{s-1}\over2}} $$
$$+2 P_{2q+{{s-1}\over2}+1} + ... + 2P_{s} \subset \Pi \cong \PP ^{n-1},$$
$$T= (n-2)Q_2 +...+ (n-2)Q_n+ H_1 + ... + H_q +
2P_{2q+1}+... + 2P_{2q+{{s-1}\over2}} $$
$$+ P'_{2q+{{s-1}\over2}+1} + ... + P'_{s}\subset \Pi \cong \PP ^{n-1},$$
and where 
 $H_i = \Lambda _i \cap \Pi \cong \PP^2$, $i=1,...,q$, and $P'_j = L_j\cap
\Pi$, $j=2q+1,...,s$. 

Since each $P'_j$ is a generic simple   point in $\Pi$, in order to compute 
$\dim (I_{W,\Pi})_{n-1} $  and $\dim (I_{T,\Pi})_{n-1}$ we first compute the dimensions of the schemes
$$W'= W- ( P'_{2q+1}+... + P'_{2q+{{s-1}\over2}}  )$$
$$= (n-2)Q_2 +...+ (n-2)Q_n + J_{H_1}^{(2)} + ... + J_{H_q}^{(2)}  + 2 P_{2q+{{s-1}\over2}+1} + ... + 2P_{s},$$
and $$T'= T-( P'_{2q+{{s-1}\over2}+1} + ... + P'_{s})$$
$$= (n-2)Q_2 +...+ (n-2)Q_n+ H_1 + ... + H_q + 2P_{2q+1}+... + 2P_{2q+{{s-1}\over2}} .$$

To compute  $\dim (I_{W'})_{n-1}  $ we apply Lemma \ref{lemmaxresiduo}, with
$$m = n-1,  \ \ \ x= q, \ \ \ y= {{s+1}\over 2}   -2q  = \left \{
\begin{matrix}    &  t+1-2q & {\rm for} & s=e  & \\
  & t^*+1-2q & {\rm for}  & s=e^* & \\
  \end{matrix} 
 \right. .$$

Similarly, in order to compute $\dim (I_{T'})_{n-1}  $ we use Lemma \ref{lemmaxtraccia}, with
$$m = n -1,\ \ \  x= q ,\ \ \ y= {{s-1}\over 2}  = \left \{
\begin{matrix}    &  t & {\rm for} & s=e  & \\
  & t^* & {\rm for}  & s=e^* & \\
  \end{matrix} 
 \right.  .$$

In the Appendix (see Section \ref{appendice}) we will check that $m,x, y$ above verify the hypotheses of  Lemmas  \ref{lemmaxresiduo} and  \ref{lemmaxtraccia}.

\par
Assume $s = e$ . 
In this case 
by Lemmas \ref{xresiduoipotesi1} and Lemma \ref{xtracciaipotesi1}
 we get
$$\dim (I_{W'})_{n-1}  =2^{n-1} -2(n-1) x - n y= 2^{n-1} -2(n-1) q - n (t+1-2q),
$$
$$\dim (I_{T'})_{n-1}  = 2^{n-1} -4x - n y= 2^{n-1} -4q - n t,
$$
Since
 $W$ is formed by  $W'$ plus ${ {s-1   }\over 2  }=t$ simple generic points, 
 and  $T$ is formed by  $T'$ plus ${{s+1}\over 2}   -2q=t+1-2q$ simple generic points
 we have
$$\dim (I_{W})_{n-1}  ={ \rm max } \{0 ; \dim (I_{W'})_{n-1} -t \},
$$
$$\dim (I_{T})_{n-1}  ={ \rm max } \{0 ; \dim (I_{T'})_{n-1} -(t+1-2q) \},
$$
and since  by Lemma \ref{lemdiseqpositive}
$$
 \dim (I_{W'})_{n-1} =2^{n-1} -2(n-1) q - n (t+1-2q)  \geq t  ,
$$
$$
 \dim (I_{T'})_{n-1} =2^{n-1} -4q - n t  \geq t+1-2q  ,
$$
we get
$$\dim (I_{W})_{n-1}  =2^{n-1} -2(n-1) q - n (t+1-2q)-t ,
$$
$$\dim (I_{T})_{n-1}  = 2^{n-1} -4q - n t- (t+1-2q) .
$$
Thus
$$\dim (I_{Z})_n \leq  \dim (I_{W})_{n-1} +\dim (I_{T})_{n-1}=
2^n   -(n+1)(2t+1) = 2^n   -(n+1)e  ,
$$
and for the case $s = e$  we are done.
\par
For $s = e^*$ (using Lemmas \ref{xresiduoipotesi2} and Lemma \ref{xtracciaipotesi2})
we get
$$\dim (I_{W'})_{n-1}  =2^{n-1} -2(n-1) x - n y= 2^{n-1} -2(n-1) q - n (t^*+1-2q),
$$
$$\dim (I_{T'})_{n-1}  = 2^{n-1} -4x - n y= 2^{n-1} -4q - n t^* .
$$

As in the previous case (since
 $W$ is formed by  $W'$ plus ${ {s-1   }\over 2  }=t^*$ simple generic points, 
 and  $T$ is formed by  $T'$ plus ${{s+1}\over 2}   -2q=t^*+1-2q$ simple generic points) 
 we have
$$\dim (I_{W})_{n-1}  ={ \rm max } \{0 ; \dim (I_{W'})_{n-1} -t ^*\},
$$
$$\dim (I_{T})_{n-1}  ={ \rm max } \{0 ; \dim (I_{T'})_{n-1} -(t^*+1-2q) \},
$$
and since  by Lemma \ref{lemdiseqnegative}
$$
 \dim (I_{W'})_{n-1} =2^{n-1} -2(n-1) q - n (t^*+1-2q)  \leq t^*  ,
$$
$$
 \dim (I_{T'})_{n-1} =2^{n-1} -4q - n t^*  \leq t^*+1-2q  ,
$$
we get
$$\dim (I_{W})_{n-1}  
=\dim (I_{T})_{n-1}  = 0 .
$$
Thus
$$\dim (I_{Z})_n \leq  \dim (I_{W})_{n-1} +\dim (I_{T})_{n-1}=0  ,
$$
and we are done also in case $s = e^*$.
\end{proof}

\section{appendix}\label{appendice}

\begin{lem}\label{xresiduoipotesi1} Let the notation be as in the proof of Theorem \ref{mainTHM2}.  Let $n \geq 5$,  $s$ odd,
$$n=4q+r, \ \ \ s =   {\left \lfloor { 2^n \over {n+1}} \right \rfloor } =2t+1 . 
$$
Then
$$\dim (I_{W'})_{n-1}  = 2^{n-1} -2(n-1) q - n (t+1-2q) 
 .$$
 \end{lem}
 
\begin{proof}
We have only to check that we may apply Lemma \ref{lemmaxresiduo}, with
$$m = n-1,  \ \ \ x= q, \ \ \ y=  t+1-2q .$$
The first case we have to consider is $n=5$: we have $s=5$, $m=4$, $x=y=1$, hence the conclusion follows by (v.1) of Lemma \ref{lemmaxresiduo}.
\par
For $n=6$  we have $s=9$, $m=5$, $x=1$ , $y=3$ and we are done by Lemma \ref{lemmaxresiduo} (vi).
\par
For  $n=7$ and $n=8$ $s$ is even, and we don't have anything to prove.
\par
For $n=9$  we have $s=51$, $m=8$, $x=2 \leq  \left \lfloor{{m-1} \over 2} \right \rfloor =3$ , $y=22$,
$x$ and $y$ even, $y \leq \left  \lfloor{  {2^{m} -2mx } \over {m+1}} \right \rfloor =24$, and we are done by  (vii) of Lemma \ref{lemmaxresiduo}.
\par
Assume $n\geq  10$. In order to apply  Lemma \ref{lemmaxresiduo}(iii) and (vii) it suffices to show that 
  there exist  $x'$ and $y'$  even such that
  
\begin{equation} \label{xresiduoxminoredi1}
0 \leq x \leq x' \leq  \left \lfloor{{m-1} \over 2} \right \rfloor   ; 
 \end{equation}
 
\begin{equation}  \label{xresiduoyminoredi1}
0 \leq y \leq y ' \leq  \left  \lfloor{  {2^{m} -2m x' } \over {m+1}} \right \rfloor   .
\end{equation}
Obviously $x\geq 0$ and it is easy to check that  also $y \geq 0 $.\par
Let
$$
{x'} = \left \{
\begin{matrix}
x & {\rm for}  & x  & {\rm even} \\
x+1 & {\rm for}  & x  & {\rm odd} \\
 \end{matrix}
  \right.
  \hskip1cm
  {y'} = \left \{
\begin{matrix}
y & {\rm for}  & y  & {\rm even} \\
y+1 & {\rm for}  & y  & {\rm odd} \\
 \end{matrix}
  \right.
 .$$

For the first inequality, we will be done if
$ x+1  \leq  \left \lfloor{{m-1} \over 2} \right \rfloor $, with $n\geq 10$. Since 
$$x+1  \leq \left  \lfloor{{m-1} \over 2}\right  \rfloor  \iff 
 2q+2 \leq n-2 \iff {{n-r} \over 2} +2 \leq n-2 \iff n \geq 8-r
$$
 and  $n\geq  10$,  then (\ref{xresiduoxminoredi1}) holds.
 \par
For the second inequality,  notice that 
$ y +1  \leq  \left  \lfloor{  {2^{m} -2m (x+1) } \over {m+1}} \right \rfloor  $ implies  (\ref{xresiduoyminoredi1}). Since
 $$ y +1  \leq  \left  \lfloor{  {2^{m} -2m (x+1) } \over {m+1}} \right \rfloor  \iff
 (t+2-2q)n \leq 2^{n-1} -2(n-1)(q+1) 
 $$
 $$
  \iff 2^{n-1}+2q+2-4n-tn \geq 0 \iff
  2^{n-1}+2-3n -{r\over 2}- {n\over 2} {\left \lfloor { 2^n \over {n+1}} \right \rfloor}   \geq 0 .
 $$
and since  for $n \geq 10$ we have
 $$2^{n-1}+2-3n -{r\over 2}- {n\over 2} {\left \lfloor { 2^n \over {n+1}} \right \rfloor} \geq
 2^{n-1}+2-3n -{3\over 2}- {n\over 2} {\left ( { 2^n \over {n+1}} \right )}$$
 $$={1\over{2n+2}} (2^n -(6n-1)(n+1)) \geq 0,
 $$
it follows that  (\ref{xresiduoyminoredi1}) holds, and we are done.
\par

\end{proof}

\begin{lem}\label{xresiduoipotesi2} Let the notation be as in the proof  of Theorem \ref{mainTHM2}.  Let $n \geq 5$,  $s$ odd,
$$n=4q+r, \ \ \ s =   {\left \lceil { 2^n \over {n+1}} \right \rceil} =2t^*+1 . 
$$
Then
$$\dim (I_{W'})_{n-1}  = 2^{n-1} -2(n-1) q - n (t^*+1-2q) 
 .$$
 \end{lem}
 
\begin{proof} As in the previous lemma,
we have only to check that we may apply Lemma \ref{lemmaxresiduo}, with
$$m = n-1,  \ \ \ x= q, \ \ \ y=  t^*+1-2q .$$
The first case we have to consider is $n=8$: we have $s=29$, $m=7$, $x=2$, $y=11$.
Let $y'=12$. Since $ x  \leq  \left \lfloor{{m-1} \over 2} \right \rfloor  =3  $
and  $ y \leq y ' \leq  \left  \lfloor{  {2^{m} -2m x} \over {m+1}} \right \rfloor =12$,
we may apply Lemma \ref{lemmaxresiduo}(iii) and (vii) and we are done.

\par
For  $n=9 $,  $s$ is even.
\par
Assume $n\geq  10$. In order to apply Lemma \ref{lemmaxresiduo}(iii) and (vii) it suffices to show that 
  there exist  $x'$ and $y'$  even such that
  
\begin{equation} \label{xresiduoxminoredi2}
0 \leq x \leq x' \leq  \left \lfloor{{m-1} \over 2} \right \rfloor   ; 
 \end{equation}
 
\begin{equation}  \label{xresiduoyminoredi2}
0 \leq y \leq y ' \leq  \left  \lfloor{  {2^{m} -2m x' } \over {m+1}} \right \rfloor   .
\end{equation}

As in the previous lemma let
$$
{x'} = \left \{
\begin{matrix}
x & {\rm for}  & x  & {\rm even} \\
x+1 & {\rm for}  & x  & {\rm odd} \\
 \end{matrix}
  \right.
  \hskip1cm
  {y'} = \left \{
\begin{matrix}
y & {\rm for}  & y  & {\rm even} \\
y+1 & {\rm for}  & y  & {\rm odd} \\
 \end{matrix}
  \right.
 .$$
In the previous lemma we already checked the inequality  (\ref{xresiduoxminoredi2}), so let us deal with the inequality  (\ref{xresiduoyminoredi2}). It is easy to check that  $y \geq 0 $, hence we will be done if
$ y +1  \leq  \left  \lfloor{  {2^{m} -2m (x+1) } \over {m+1}} \right \rfloor  $.
We have
 $$ y +1  \leq  \left  \lfloor{  {2^{m} -2m (x+1) } \over {m+1}} \right \rfloor  \iff
 (t^*+2-2q)n \leq 2^{n-1} -2(n-1)(q+1) 
 $$
 $$
  \iff 2^{n-1}+2q+2-4n-t^*n \geq 0 \iff
  2^{n-1}+2-3n -{r\over 2}- {n\over 2} {\left \lceil { 2^n \over {n+1}} \right \rceil}   \geq 0 .
 $$
 Since for $n \geq 10$ we have
 $$2^{n-1}+2-3n -{r\over 2}- {n\over 2} {\left \lceil { 2^n \over {n+1}} \right \rceil} \geq
 2^{n-1}+2-3n -{3\over 2}- {n\over 2} {\left ( { 2^n \over {n+1}} +1\right )}$$
 $$={1\over{2n+2}} (2^n -6n(n+1)) \geq 0,
 $$
then it follows that  (\ref{xresiduoyminoredi2}) holds, and we are done.
\end{proof}

\begin{lem} \label{xtracciaipotesi1} Let the notation be as in the proof of Theorem \ref{mainTHM2}.  Let $n \geq 5$,  $s$ odd,
$$n=4q+r, \ \ \ s =   {\left \lfloor { 2^n \over {n+1}} \right \rfloor } =2t+1 . 
$$
Then
$$\dim (I_{T'})_{n-1}  = 2^{n-1} -4 q - n t
 .$$
\end{lem}

\begin{proof}
We want to check that we may apply Lemma \ref{lemmaxtraccia}, with
$$m = n-1,  \ \ \ x= q, \ \ \ y=  t .$$

For $n=5$  we have $s=5$, $m=4$, $x=1$, $y=2$, hence we are in case (iv) of Lemma \ref{lemmaxtraccia}.
\par
For $n=6$  we have $s=9$, $m=5$, $x=1$ , $y=4$ and we are done by Lemma \ref{lemmaxtraccia} (ii) and (v).
\par
For  $n=7$ and $n=8$ $s$ is even.
\par
Assume $n\geq  9$. In order to apply  (ii) and (vi)  of Lemma \ref{lemmaxtraccia} it suffices to show that    there exist  $x'$ and $y'$  even such that
  
  \begin{equation} \label{xtracciaxminoredi1}
0 \leq x \leq x' \leq  \left \lfloor{{m-1} \over 2} \right \rfloor   ; 
 \end{equation}

 \begin{equation} \label{xtracciayminoredi1}
  0 \leq y   \leq y' \leq  \left  \lfloor{  {2^{m} - 4x' } \over {m+1}} \right \rfloor   . 
\end{equation}

As usual let
$$
{x'} = \left \{
\begin{matrix}
x & {\rm for}  & x  & {\rm even} \\
x+1 & {\rm for}  & x  & {\rm odd} \\
 \end{matrix}
  \right.
  \hskip1cm
  {y'} = \left \{
\begin{matrix}
y & {\rm for}  & y  & {\rm even} \\
y+1 & {\rm for}  & y  & {\rm odd} \\
 \end{matrix}
  \right.
 .$$
For the proof of (\ref{xtracciaxminoredi1})   see the proof given for (\ref{xresiduoxminoredi1}) 
in Lemma \ref{xresiduoipotesi1}.
\par
We easily get $y \geq 0 $.
Since
$ y +1  \leq  \left  \lfloor{  {2^{m} -4 (x+1) } \over {m+1}} \right \rfloor  $ implies  (\ref{xtracciayminoredi1}), we prove that this inequality holds.
We have
 $$ y +1  \leq  \left  \lfloor{  {2^{m} -4 (x+1) } \over {m+1}} \right \rfloor  \iff
 (t+1)n \leq 2^{n-1} -4(q+1) 
 $$
 $$
  \iff 2^{n-1} -2n -4+r   -{n\over 2}      {\left \lfloor { 2^n \over {n+1}} \right \rfloor}   + {n\over 2}     \geq 0 . $$
 Since for $n \geq 9$ we have
 $$
 2^{n-1} -2n -4+r   -{n\over 2}      {\left \lfloor { 2^n \over {n+1}} \right \rfloor}  + {n\over 2}     \geq
    {1 \over {2n+2}}(2^{n}-( 3n +8)(n+1))     \geq 0,
 $$
 then it follows that  (\ref{xtracciayminoredi1}) holds, and we are done.
\end{proof}

\begin{lem} \label{xtracciaipotesi2} Let the notation be as in the proof  of Theorem \ref{mainTHM2}.  Let $n \geq 5$,  $s$ odd,
$$n=4q+r, \ \ \ s =   {\left \lceil { 2^n \over {n+1}} \right \rceil} =2t^*+1 . 
$$
Then
$$\dim (I_{T'})_{n-1}  = 2^{n-1} -4 q - n t^*
 .$$
\end{lem}
\begin{proof}
As in the previous lemma, we want to check that we may apply Lemma \ref{lemmaxtraccia}  with
$$m = n-1,  \ \ \ x= q, \ \ \ y=  t ^*.$$

For $n=5, 6, 7,9$,  $s$ is not odd.\par
For   $n=8$ we have $s=29$, $m=7$, $x=2\leq  \left \lfloor{{m-1} \over 2} \right \rfloor =3$, 
$y=14    \leq  \left  \lfloor{  {2^{m} -4x} \over {m+1}} \right \rfloor   =15,$ and we are done by
(vi)  of Lemma \ref{lemmaxtraccia}
\par
Assume $n\geq  10$. In order to apply    Lemma \ref{lemmaxtraccia} (ii) and (vi) we show that   the even integers
$$
{x'} = \left \{
\begin{matrix}
x & {\rm for}  & x  & {\rm even} \\
x+1 & {\rm for}  & x  & {\rm odd} \\
 \end{matrix}
  \right.
  \hskip1cm
  {y'} = \left \{
\begin{matrix}
y & {\rm for}  & y  & {\rm even} \\
y+1 & {\rm for}  & y  & {\rm odd} \\
 \end{matrix}
  \right.
 $$
  are such that
 \begin{equation} \label{xtracciaxminoredi2}
0 \leq x \leq x' \leq  \left \lfloor{{m-1} \over 2} \right \rfloor   ; 
 \end{equation}
 
\begin{equation}  \label{xtracciayminoredi2}
0 \leq y \leq y ' \leq  \left  \lfloor{  {2^{m} -4x'} \over {m+1}} \right \rfloor   .
\end{equation}

For the proof of (\ref{xtracciaxminoredi2})   see the one given for (\ref{xresiduoxminoredi1}) 
in Lemma \ref{xresiduoipotesi1}.
\par
Obviously $y \geq 0 $.
Since
$ y +1  \leq  \left  \lfloor{  {2^{m} -4 (x+1) } \over {m+1}} \right \rfloor  $ implies  (\ref{xtracciayminoredi2}), we prove that this last inequality holds.
We have
 $$ y +1  \leq  \left  \lfloor{  {2^{m} -4 (x+1) } \over {m+1}} \right \rfloor  \iff
 (t^*+1)n \leq 2^{n-1} -4(q+1) 
 $$
 $$
  \iff 2^{n-1} -2n -4+r   -{n\over 2}     {\left \lceil { 2^n \over {n+1}} \right \rceil}   + {n\over 2}     \geq 0 . $$
 Since for $n \geq 10$ we have
 $$
 2^{n-1} -2n -4+r   -{n\over 2}      {\left \lceil { 2^n \over {n+1}} \right \rceil}  + {n\over 2}     \geq
    {1 \over {2n+2}}(2^{n}-4( n +2)(n+1))     \geq 0,
 $$
 then it follows that  (\ref{xtracciayminoredi2}) holds, and we are done.
\end{proof}

\begin{lem} \label{lemdiseqpositive}
Let $n \geq 5$, $0 \leq r \leq 3$,  $s$ odd,
$$n=4q+r, \ \ \ s =   {\left \lfloor { 2^n \over {n+1}} \right \rfloor } =2t+1 . 
$$
Then \par
$
{\rm (i) } \ 
2^{n-1} -2(n-1) q - n (t+1-2q) - t  \geq 0 ;
$
\par
$
{\rm (ii) } \ 
2^{n-1} -4q - n t  -( t+1-2q) \geq 0   .
$
\end{lem}
\begin{proof} 
Notice that, since $ {\left \lfloor { 2^n \over {n+1}} \right \rfloor }$ is odd, $n+1$ does not divides $2^n$.
Let $h,k \in \Bbb N$ be such that
$$ 2^n = (n+1) h +  k,  \ \ \ \ \   1 \leq  k \leq n.
$$

(i) We have
$$
2^{n-1} -2(n-1) q - n (t+1-2q) - t  =
2^{n-1}+2 q - (n+1)t-n =
$$
$$
{1\over2} ((n+1) h +  k +  (n-r) -(n+1) (h-1) - 2n)=
{1\over2} ( k -r+1 ).
$$
Since $k \geq 1 $, then for $r\leq 2$ we are done.
If $r=3$,  then $4$ divides $n+1$, hence $k \neq1 $ and the conclusion follows.

(ii) We have
$$2^{n-1} -4q - n t  -( t+1-2q) =
2^{n-1}-2 q - (n+1)t-1 =
$$
$${1\over2} ((n+1) h +  k -  (n-r) -(n+1) (h-1) - 2)= k +r -1 .
$$ 
Since $k \geq 1 $, then (ii) holds.
\end{proof}

\begin{lem} \label{lemdiseqnegative}
Let $n \geq 5$, $0 \leq r \leq 3$,  $s$ odd,
$$n=4q+r, \ \ \ s =   {\left \lceil { 2^n \over {n+1}} \right \rceil} =2t^*+1 . 
$$
Then
\par
{\rm (i) }$
2^{n-1} -2(n-1) q - n (t^*+1-2q) - t^* \leq 0 .
$
\par
{\rm (ii) }$
2^{n-1} -4q - n t^*  -( t^*+1-2q )\leq 0.
$
\end{lem}

\begin{proof} 
Let $h,k \in \Bbb N$ be as in the proof of Lemma \ref{lemdiseqpositive}:
$$ 2^n = (n+1) h +  k,  \ \ \ \ \   1 \leq k \leq n.
$$

(i)
$$
2^{n-1} -2(n-1) q - n (t^*+1-2q) - t ^* =
2^{n-1}+2 q - (n+1)t^*-n =
$$
$$
{1\over2} ((n+1) h +  k +  (n-r) -(n+1) h - 2n)=
{1\over2} ( k -r-n ) \leq {1\over2} (  -r )  \leq 0.
$$

(ii) 
$$
2^{n-1} -4q - n t^*  -( t^*+1-2q )=
2^{n-1} -2 q - (n+1)t^*-1 =
$$
$$
{1\over2} ((n+1) h +  k -  (n-r) -(n+1) h - 2)=
{1\over2} (   k - n+r - 2).
$$
Since $k \leq n$, for $r\leq 2$ we are done. 
If $r=3$,  then  $n=4q+3$ is odd, moreover  
$2^{n} =4(q+1)h + k $, hence $k  $ is even. It follows that 
$k \leq n-1$ and the conclusion follows.
\end{proof}

\bibliographystyle{alpha}
\bibliography{marvi.bib}

\end{document}